\documentclass{compositio}
\usepackage{amssymb,bbm}
\usepackage{amscd,amsfonts,amsmath,amssymb,bbm, dsfont}
\usepackage{enumerate,epsf,fancyhdr,float,graphicx,tabularx}
\usepackage{latexsym,mathrsfs,multirow}
\usepackage{wasysym}
\usepackage{xypic}
\usepackage[all]{xy}\usepackage[OT2,T1]{fontenc}
\usepackage[modulo,mathlines,displaymath, running]{lineno}
\usepackage{amsmath,mathrsfs,enumerate,wasysym}
\usepackage{xypic,hhline}
\usepackage[all]{xy}
\usepackage[OT2,T1]{fontenc}
\usepackage[modulo,mathlines,displaymath]{lineno}
\usepackage{tikz,tikz-cd}
\usepackage{dashrule}
\usetikzlibrary{matrix}
\usepackage[modulo,mathlines,displaymath]{lineno}

\newtheorem{theorem}{Theorem}[section]
\DeclareSymbolFont{cyrletters}{OT2}{wncyr}{m}{n}\DeclareMathSymbol{\Sha}{\mathalpha}{cyrletters}{"58}

\renewcommand{\phi}{{\varphi}}

\newcommand{\Zp}{\mathbf{Z}_p}

\renewcommand{\geq}{\geqslant}
\renewcommand{\leq}{\leqslant}

\newcommand{\Ehat}{{\widehat{E}}}

\newcommand{\smat}[1]{\left( \begin{smallmatrix} #1 \end{smallmatrix} \right)}

\newcommand{\links}{\left(\begin{array}{cc}}
\newcommand{\rechts}{\end{array}\right)}
\newcommand{\bai}{\left[\begin{array}{cc}}
\newcommand{\dai}{\end{array}\right]}
\newcommand{\hidari}{\left(\begin{array}{c}}
\newcommand{\migi}{\end{array}\right)}

\newcommand{\C}{\mathbb{C}}
\newcommand{\DD}{\mathbb{D}}

\newcommand{\Q}{\mathbb{Q}}

\newcommand{\Z}{\mathbb{Z}}

\newcommand{\gp}{{\mathfrak p}}
\newcommand{\gq}{{\mathfrak q}}

\newcommand{\dg}{{\mathfrak g}}
\newcommand{\gf}{{\mathfrak f}}

\newcommand{\gm}{{\mathfrak m}}

\newcommand{\bz}{\mathbf{z}}
\newcommand{\bH}{\mathbf{H}}

\newcommand{\calF}{\mathcal{F}}
\newcommand{\calG}{\mathcal{G}}

\newcommand{\calL}{\mathcal{L}}

\newcommand{\calP}{\mathcal{P}}

\newcommand{\F}{{\mathcal F}_{ss}}
\newcommand{\G}{{\mathcal G}}
\renewcommand{\H}{\mathcal{H}}

\newcommand{\X}{{\mathcal X}}
\newcommand{\tor}{{\mathrm{tor}}}

\newcommand{\Log}{\mathcal Log}

\newcommand{\Qbar}{\overline{\Q}}
\newcommand{\Zhat}{\hat{\Z}}

\newcommand{\Reg}{\operatorname{Reg}}

\DeclareMathOperator{\Frac}{Frac}

\DeclareMathOperator{\proj}{\mathrm proj}
\newcommand{\Iw}{\mathrm{Iw}}

\DeclareMathOperator{\rk}{rk}

\DeclareMathOperator{\tr}{tr}

\newcommand{\GalQ}{{\Gal}(\Qbar/\Q)}

\newcommand{\Spec}{\operatorname{Spec}}

\newcommand{\Gal}{\operatorname{Gal}}

\newcommand{\cor}{\operatorname{cor}}

\newcommand{\Hom}{\operatorname{Hom}}

\newcommand{\Aut}{\operatorname{Aut}}

\newcommand{\Char}{\operatorname{Char}}
\newcommand{\GL}{\operatorname{GL}}
\newcommand{\ord}{\operatorname{ord}}
\newcommand{\col}{\operatorname{Col}}
\newcommand{\loc}{\operatorname{loc}}
\newcommand{\Wan}{\operatorname{Wan}}

\newcommand{\im}{\operatorname{Im}}
\newcommand{\Sel}{\operatorname{Sel}}

\renewcommand{\O}{\mathcal{O}}

\newcommand{\directsum}{\oplus} 
\newcommand{\injects}{\hookrightarrow}

\newcommand{\tensor}{\otimes} 

\renewcommand{\contentsname}{Contents\\{\footnotesize\normalfont(A table
of contents should normally not be included)}}

\newtheorem{auxiliary proposition}[theorem]{Auxiliary Proposition}

\newtheorem{choice}[theorem]{Choice}

\newtheorem{conjecture}[theorem]{Conjecture}
\newtheorem{convention}[theorem]{Convention}
\newtheorem{corollary}[theorem]{Corollary}

\newtheorem{definition}[theorem]{Definition}

\newtheorem{lemma}[theorem]{Lemma}
\newtheorem{main conjecture}[theorem]{Main Conjecture}
\newtheorem{main theorem}[theorem]{Main Theorem}
\newtheorem{modesty proposition}[theorem]{Modesty Proposition}

\newtheorem{open problem}[theorem]{Open Problem}

\newtheorem{proposition}[theorem]{Proposition}

\newtheorem{remark}[theorem]{Remark}

\newtheorem{convergence lemma}[theorem]{Convergence Lemma}
\newtheorem{corrected lemma}[theorem]{Corrected Lemma}
\newtheorem{growth lemma}[theorem]{Growth Lemma}
\newtheorem{coefficient lemma}[theorem]{Integrality Lemma}
\newtheorem{interpolation lemma}[theorem]{Interpolation Lemma}
\newtheorem{kernel lemma}[theorem]{Kernel Lemma}
\newtheorem{limit lemma}[theorem]{Limit Lemma}
\newtheorem{tandem lemma}[theorem]{Modesty Lemma}
\newtheorem{zero-finding lemma}[theorem]{Zero-Finding Lemma}

\newcommand{\badLs}{\links L_{\alpha\alpha} & L_{\beta\alpha} \\ L_{\alpha\beta} & L_{\beta\beta} \rechts} 
\newcommand{\goodLs}{\links L^{\sharp\sharp} & L^{\flat\sharp} \\ L^{\sharp\flat} & L^{\flat\flat} \rechts
}
\begin{document}

\title{The Iwasawa Main Conjecture for elliptic curves at odd supersingular primes}
\author{Florian Sprung}
\email{fsprung@math.princeton.edu}
\address{Princeton University and the Institute for Advanced Study, School of Mathematics, 1 Einstein Dr or 304 Fine Hall, Princeton, NJ 08540, USA}

\thanks{This material is based upon work supported by the National Science Foundation under agreement No. DMS-1128155. Any opinions, findings and conclusions or recommendations expressed in this material are those of the author and do not necessarily reflect the views of the National Science Foundation.}

\begin{abstract}
In this paper, we prove the Iwasawa main conjecture for elliptic curves at an odd supersingular prime $p$. Some consequences are the $p$-parts of the leading term formulas in the Birch and Swinnerton-Dyer conjectures for analytic rank $0$ or $1$.\end{abstract}

\maketitle


\pagestyle{plain}

\section{Introduction}
\label{sec:introduction}
Iwasawa theory gained a place in the study of elliptic curves in the $1970'$s in the work of Mazur. Given an elliptic curve $E$ over $\Q$ and a prime $p$, the goal (or main conjecture) of this program is to relate a $p$-adic $L$-function attached to $E$ with a Selmer group, which contains information about the arithmetic of $E$ in subfields of the cyclotomic $\Z_p$-extension $\Q_\infty$ of $\Q$. Understanding why these objects are closely related is difficult because they are very different in nature: Selmer groups are algebraic, while $p$-adic $L$-functions are analytic. The main conjecture has important consequences for the Birch and Swinnerton-Dyer conjectures for $E$. 

When $p$ is of good ordinary reduction (i.e. $p\nmid a_p:=p+1-\#E(\mathbb{F}_p)$), this program is now largely complete. Mazur and Swinnerton-Dyer introduced the $p$-adic $L$-function in \cite{msd}, while Mazur studied the corresponding Selmer group in \cite{mazur}. In the complex multiplication (CM) case, Rubin \cite{rubin} proved the main conjecture by finding a suitable Euler system. Kato proved half the main conjecture in the non-CM case by finding his zeta element Euler system, showing that the $p$-adic $L$-function is included in the characteristic ideal associated to the Selmer group \cite{kato}. Finally, Skinner and Urban \cite{SU} gave a converse to Kato's theorem via a GU(2,2) Eisenstein series method, settling the main conjecture, for a large class of elliptic curves.  

The supersingular case (i.e. when $p|a_p$) has been more challenging. The main obstacle was that the objects are not well-behaved: The $p$-adic $L$-function is not an element of the Iwasawa algebra, and correspondingly the Selmer group is not a torsion Iwasawa module. Perrin-Riou \cite{PR} and Kato \cite{kato} made important progress in our understanding of the supersingular case, and formulated main conjectures. For a formulation of the main conjecture in the spirit of the ordinary case, Pollack (when $a_p=0$ \cite{pollack}) found a pair of $p$-adic $L$-functions $L^\sharp$ and $L^\flat$ inside the Iwasawa algebra (for the general $p|a_p$ case, see \cite{shuron}). Kobayashi (again when $a_p=0$ \cite{kobayashi}) found two corresponding submodules $\Sel^\sharp(E/\Q_\infty)$ and $\Sel^\flat(E/\Q_\infty)$ of the classical Selmer group, and formulated a pair of main conjectures in terms of these objects (for $p|a_p$, see \cite{shuron}). Pollack and Rubin \cite{P+R} settled the case of complex multiplication by adapting Rubin's Euler system. Relying on Kato's Euler system, half of the main conjecture was proved \cite{kobayashi,shuron}, i.e. $L^{\sharp/\flat}$ is each included in the characteristic ideal coming from $\Sel^{\sharp/\flat}(E/\Q_\infty)$. In the case $a_p=0$, Wan \cite{wan} settled the main conjecture by giving a converse to this theorem via a GU(3,1) Eisenstein series method. 

The purpose of this paper is to prove the main conjecture in the general supersingular case:

\begin{theorem}Let $E/\Q$ be an elliptic curve and $p>2$ a prime of supersingular reduction. Assume that $E$ has square-free conductor.
Then $L^{\sharp/\flat}$ are each characteristic power series of the Iwasawa module $\Hom(\Sel^{\sharp/\flat}(E/\Q_\infty),\Q_p/\Z_p)$.
\end{theorem}

Our corollaries concern the leading term formula in the Birch and Swinnerton-Dyer conjecture.

\begin{corollary}Assume that $L(E,1)\neq0$. Then under the assumptions of the theorem, we have

$$\left|\frac{L(E,1)}{\Omega}\right|_p=\left|\#\Sha(E/\Q)\prod_l c_l\right|_p$$

\end{corollary}

Here, $\Sha(E/\Q)$ is the Tate--Shafarevich group of $E/\Q$, $c_l$ are the Tamagawa numbers, and $\Omega$ is the N\'{e}ron period.
Combined with the corresponding result in the ordinary case \cite[Theorem 3.35]{SU}, this gives the leading term formula up to powers of $2$ and bad primes.

\begin{corollary}Suppose that $\ord_{s=1}L(E,s)=1$. Then under the assumption of the theorem,

$$\left|\frac{L'(E,1)}{\Reg(E/\Q)\Omega}\right|_p=\left|\frac{\#\Sha(E/\Q)\prod_l c_l}{\#E(\Q)^2_\tor}\right|_p$$

\end{corollary}

Here, $\Reg(E/\Q)$ denotes the regulator of $E$, and $E(\Q)_\tor$ is the torsion part of $E(\Q)$. This corollary follows from our theorem combined with forthcoming work of Jetchev, Skinner, and Wan \cite{jsw}, or with the discussion of \cite[Corollary 1.3]{kobayashigz}.

Note that the Hasse--Weil bound $|a_p|\leq2\sqrt{p}$ (see e.g. \cite[Chapter 5]{silverman}) forces $p=2$ or $p=3$ when $p|a_p\neq0$, so that in this case $p=3$, since $p$ is odd. This still covers infinitely many elliptic curves, and we have worked out our methods so that they should generalize to general modular forms of weight two, at least when $\ord_p(a_p)\geq\frac{1}{2p}$: Put differently, the assumption $p=3$ is not used. About $11\%$ of elliptic curves with good reduction at $3$ have $a_3=\pm3$.

The global strategy of the proof is to establish an equivalence\begin{footnote}{The use of `equivalence' and `equivalent' in the introduction is slightly inaccurate. The equivalence is only given up to powers of $p$ and \textit{exceptional primes} (Definition \ref{exceptional}), but this defect is fixable for our purposes.}\end{footnote} of three sets of main conjectures over an auxiliary imaginary quadratic field $K$. The first set (four $\sharp/\flat$-$\sharp/\flat$ main conjectures) implies the $\sharp/\flat$ main conjectures over $\Q$ which we want to prove. We show that these $\sharp/\flat$-$\sharp/\flat$ main conjectures are equivalent to a second set of main conjectures formulated in terms of new integral Euler systems $\Delta_\sharp$ and $\Delta_\flat$ which we construct from the Euler systems $\Delta_\alpha$ and $\Delta_\beta$ of Kings--Loeffler--Zerbes. We show that this second set is itself equivalent to a third set, consisting of just one Greenberg-type main conjecture. The Greenberg-type main conjecture is amenable to the $GU(3,1)$-Eisenstein series method of Wan, who has proved one inclusion \cite{wanrankin}. We can then carry this inclusion over to the other two sets of main conjectures over $K$, which gives us a result towards half of the $\sharp/\flat$-main conjecture over $\Q$. Combining this with the known inclusion \cite[Theorem 7.16]{shuron} completes the proof. 

The above situation has an analogue in classical Iwasawa theory. The $\sharp/\flat$-$\sharp/\flat$ main conjectures are analogues of the classical main conjecture with the Kubota--Leopoldt $p$-adic $L$-function, while the main conjecture in terms of $\Delta_\sharp$ and $\Delta_\flat$ is an analogue to the formulation of the classical main conjecture in terms of cyclotomic units. The Greenberg-type main conjecture is (more than just) an analogue to Greenberg's formulation of the classical main conjecture found in \cite{greenberg}.

This strategy generalizes that of Wan for the case $a_p=0$. Initially, a generalization to the full supersingular case was assumed to be not hard, and it was hoped that some of the important arguments would be as similar to the ordinary case as for the case $a_p=0$ (e.g. manifested in the fact that the $\sharp$- and $\flat$-main conjectures can be proved \textit{separately}). However, many of the arguments break down when $a_p\neq0$, and the $\sharp/\flat$ main conjectures need to be handled \textit{together}. Another difficulty was that much of the background theory needed for \cite{wan} assumed $a_p=0$. We overcome this difficulty by generalizing the appropriate theories to general supersingular primes. 

More precisely, we choose a quadratic imaginary field in which $(p)=\gp\gq$ splits. The $\sharp/\flat$-$\sharp/\flat$ main conjectures are generalizations of the $\pm$-and-$\pm$ main conjectures of B.D. Kim (who assumed $a_p=0$). The essential ingredient in their formulation is a pair of $\sharp/\flat$-Coleman maps which give rise to a four-tuple of appropriate modified $\sharp/\flat$-$\sharp/\flat$ Selmer groups, whose Pontrjagin duals $\X^{**}$ (with each $*\in\{\sharp,\flat\})$ are torsion as $\Lambda_K=\Z_p[[X,Y]]$-modules. The main conjecture relates these $\X^{**}$ to Lei's four $\sharp/\flat$-$\sharp/\flat$ $p$-adic $L$-functions, which are elements\begin{footnote}{We prove this via a higher-dimensional version of the Weierstra\ss \text{ } preparation theorem. This may be a cute lemma (Lemma \ref{cutelemma}) for the interested reader.}\end{footnote} of $\Lambda_K$.    

Lei's four $p$-adic $L$-functions appear in a matrix in the left part of the three-part diagram below which features the analytic objects of the three equivalent main conjectures, below which are the formal statement of some of the main conjectures.
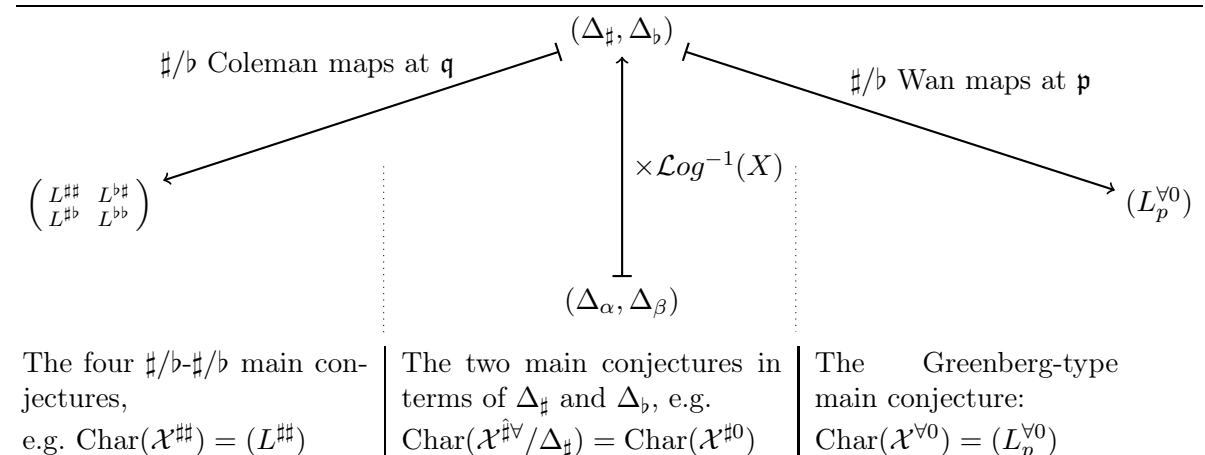
\usetikzlibrary{positioning}
\begin{tikzpicture}
\matrix [column sep=7mm, row sep=5mm] {\hline&&&&&&&
  \node (ab) {$(\Delta_\sharp,\Delta_\flat)$}; \\
 & &&&\\
 \\
\node (0) {$\smat{L^{\sharp\sharp} &L^{\flat\sharp} \\ L^{\sharp\flat} & L^{\flat\flat}}$}; &&&&\node (upperleft) {};&&&&&&&&&&&\node (2) {$(L_p^{\forall0})$}; \\
&&&&\node (left) {};&&&\node (1)  {$(\Delta_\alpha,\Delta_\beta)$};\\
};

\draw[<-|, thick] (0) to node[pos=0.7,above] {\hspace{-35mm}$\sharp/\flat$ Coleman maps at $\gq$}(ab);
\draw[|->, thick] (1) to node[right]{$\times\Log^{-1}(X)$} (ab);
\draw[<-|, thick] (2) to node[pos=0.6,above]{\hspace{30mm}$\sharp/\flat$ Wan maps at $\gp$} (ab);
\draw[ dotted] (-2.97,0) -- (-2.97,-2.25);
\draw[ dotted] (2.45,0) -- (2.45,-2.2);
\end{tikzpicture}

\begin{tabularx}{440pt}{p{4.55cm}|p{5cm}|p{4cm}}
The four $\sharp/\flat$-$\sharp/\flat$ main conjectures, & The two main conjectures in terms of $\Delta_\sharp$ and $\Delta_\flat$, e.g.&The Greenberg-type main conjecture:\\
e.g. $\Char(\X^{\sharp\sharp})=(L^{\sharp\sharp})$&$\Char(\X^{\hat{\sharp}\forall}/\Delta_\sharp)=\Char(\X^{\sharp0})$&$\Char(\X^{\forall0})=(L_p^{\forall0})$\\
\hline
\end{tabularx}

The Euler systems $\Delta_\alpha$ and $\Delta_\beta$ constructed in \cite{klz} live in a ring with bad growth properties in the variable $X$ corresponding to $\gp$. We fix this deficit by factoring out an explicit $2\times2$-matrix $\Log(X)$ responsible for the bad growth properties, and obtain an integral version $(\Delta_\sharp,\Delta_\flat)$. The $\sharp$- and $\flat$- Coleman maps $\col^{\sharp/\flat}$ send (the local image at $\gq$ of) $\Delta_\sharp$ and $\Delta_\flat$ to Lei's four $p$-adic $L$-functions, up to some controllable constant. Similarly, we construct $\sharp$- and $\flat$- Wan maps $\Wan^{\sharp/\flat}$ which map each of (the local images at $\gp$ of) $\Delta_\sharp$ and $\Delta_\flat$ to a $p$-adic $L$-function $L_p^{\forall0}$, up to some controllable elements of $\Lambda_K$.  The $p$-adic $L$-function $L_p^{\forall0}$ appears in the formulation of the Iwasawa main conjecture \`{a} la Greenberg \cite{greenberg}, purely in terms of the representations corresponding to $E$ and a Hida family of normalized CM forms coming from characters of $\Gal(K_\infty/K)$. The maps $\Wan^{\sharp/\flat}$ at $\gp$ can be thought of as dual to $\col^{\sharp/\flat}$: they describe\begin{footnote}{$\Wan^\flat$ restricts to a map $LOG^+$ constructed by Wan \cite{wan} in the $a_p=0$ case. Our construction of the Wan maps are backwards from Wan's, starting with this description of $Ker\col^{\sharp/\flat}$ rather than having it as a consequence.}\end{footnote} the kernel of $\col^{\sharp/\flat}$ at $\gp$. 

The essential properties of the algebraic objects are the local conditions. Roughly speaking, $\X^{\hat{\sharp}\forall}$ (together with $\Delta_\sharp$) and $\X^{\sharp0}$ are rich enough to understand both $\X^{\sharp\sharp}$ and $\X^{\forall0}$.

\begin{tabularx}{323.55pt}{c|c|c|c|c|}
$\text{(Dual of) Selmer group}$ & $\X^{\sharp\sharp}$ & $\X^{\sharp0}$& $\X^{\hat{\sharp}\forall}$&$\X^{\forall0}$\\
\hline
local condition at $\gp$ &$\sharp$&$\sharp$&dual of ${\sharp}$& (everything) \\
\hline
local condition at $\gq$ &$\sharp$&0&(everything)&0\\
\end{tabularx}

To connect the ($\Delta_\sharp,\Delta_\flat$)-main conjectures with a main conjecture of the form $\Char{\X}=(L_p)$, we use methods similar to \cite[Section 6]{kurihara} and consider an exact sequence of $\Lambda_K$-modules  
$$0\rightarrow\X^{\hat{\sharp}\forall}/\Delta_\sharp\rightarrow\Lambda_K/(\text{some }L_p)\rightarrow\text{(some $\X$)} \rightarrow \X^{\sharp0}\rightarrow 0,$$

which shows $\Char(\X^{\hat{\sharp}\forall}/\Delta_\sharp)=\Char(\X^{\sharp0})$ if and only if $\Char(\Lambda_K/(L_p))=\Char(\X).$

\begin{tabularx}{440pt}{|c|X|}

\multicolumn{2}{l}{\it Notation introduced in section 2.1}\\
\hline
$\Q_\infty$ & the cyclotomic $\Z_p$-extension of $\Q$\\
$\Phi_{p^n}(Y)$ & the $p^n$th cyclotomic polynomial $\sum_{i\geq 0}^{p-1}Y^{p^{n-1}i}$ \\
$\zeta_{p^n}$ & a primitive $p^n$th root of unity \\
$\gamma$ & a topological generator of $\Z_p$, or its image in $\Z_p/p^n\Z_p$  \\
$\alpha$ and $\beta$ & the roots of the Hecke polynomial $Y^2-a_pY+p$ \\
$\Omega_E$ & the N\'{e}ron period of $E$  \\
$T$ & the Tate-module for $E$  \\
$V$ & $\Q_p\tensor T$\\
$\Q_{p,n}$ & the $n$th layer in the cyclotomic $\Z_p$-extension of $\Q_p$\\
\hline
\end{tabularx}\\

\begin{tabularx}{440pt}{|c|X|c|}
\multicolumn{2}{l}{\it Notation introduced in section 2.2}\\
\hline

$K$ (and $\gp,\gq$) & a quadratic imaginary field of class number coprime to $p$ in which $p$ splits as $\gp\gq$ \\
$K_\infty$&the $\Z_p^2$-extension of $K$\\
$\Omega_E^+$ and $\Omega_E^-$&the real and imaginary periods of ${E}$ (see e.g. \cite[Section 9.2]{skinnerzhang})\\
$\Lambda_K$&$\Z_p[[\Gal(K_\infty/K)]]\cong\Z_p[[X,Y]].$ We identify $1+X$ with a generator of the Galois group of the $\gp^\infty$ ray class field $\Gal(K(\gp^\infty)/K)$ (and do the same for $1+Y$ and $\Gal(K(\gq^\infty)/K)$).\\
$k,O_k$ & an unramified extension of $\Q_p$ and its ring of integers \\
$k^m$ & the degree $p^m$ unramified extension of $\Q_p$\\
$k_n,\gm_n$ & the $\Z/p^n\Z$-subextension of $k(\zeta_{p^{n+1}})$, the maximal ideal in its integer ring.\\
$k_{n,m},\gm_{n,m}$ & the objects above with $k=k^m$.\\
$\Lambda_{n,m}$ & $\Z_p[\Gal(k_{n,m}/\Q_p)]$\\
$\Lambda$ & $\Z_p[[X]]\cong\Z_p[[\Gal(\Q_\infty/\Q)]]$\\
\hline
\end{tabularx}\\

\begin{tabularx}{440pt}{|c|X|c|}
\multicolumn{2}{l}{\it Notation introduced in section 2.3}\\
\hline
$U_\gp$ & the group $\Gal(K_{\infty,\gp}/K_{cyc,\gp})\cong \Z_p$; $K_{cyc}$ is the {cyclotomic $\Z_p$-extension of $K$}\\
$U$ & the unramified variable, i.e. $1+U$ is identified with a topological generator of $U_\gp$\\
$\Lambda_K^*$ & the Pontryagin dual of $\Lambda_K$\\
$\Psi$&the character $\Gal(\overline{K}/K)\rightarrow \Gal(K_\infty/K)\rightarrow\Lambda_K^\times$\\
$\Lambda_K^*(\Psi)$ & $\Lambda_K^*$ twisted by $\Psi$\\
\hline
\end{tabularx}\\

\begin{tabularx}{440pt}{|c|X|c|}

\multicolumn{2}{l}{\it Notation introduced in section 2.4}\\
\hline
$\gamma_1,\gamma_2,\cdots, \gamma_{p^t}$&Representatives of $\Gal(K_\infty/K)/D_\gp$, where $D_\gp$ is the decomposition group of $\gp$ in $\Gal(K_\infty/K)$.\\
$\pi_\alpha$ and $\pi_\beta$ & The projection maps from $\DD_{cris}(V_f)$ to its $\alpha$-(resp.$\beta$-) eigenspace for the Frobenius action with respect to the basis given by the N\'{e}ron differential $\omega_E$. \\
$\pi_{\calF_\dg^+}$&Projection onto $\calF_\dg^+$. ($\calF_\dg^\pm$ are defined in Section 2.4.) $\dg$ is the Hida algebra of CM forms of $\Gal(K_\infty/K)$ with coefficient ring $\Z_p[[T]]$, where $1+T$ i identified with a topological generator of $\Gal(K_\infty/K_{cyc})$\\
$\Gamma_{n,\gq}$&the Galois group of the totally ramified $\Z/p^n\Z$-extension of $K_\gq$\\
\hline
\end{tabularx}

\section{Main Conjectures}
\label{sec:mainconjectures}
Iwasawa main conjectures relate analytic objects to algebraic objects. The goal of this section is to state four such main conjectures. The first one (subsection \ref{subsec:sharpflat}) is the main conjecture concerning the cyclotomic $\Z_p$-extension of $\Q$, which we want to prove. The other three main conjectures concern $\Z_p$-extensions of imaginary quadratic fields. They are the $\sharp/\flat$-$\sharp/\flat$ main conjecture (subsection \ref{subsec:twovariable}), the Greenberg-type main conjecture (subsection \ref{subsec:greenberg}), and the main conjecture involving $\sharp/\flat$-Beilinson-Flach elements (subsection \ref{subsec:beilinsonflach}). In section \ref{sec:connecting}, we will prove that the $\sharp/\flat$-$\sharp/\flat$ main conjecture and the Greenberg-type main conjecture are equivalent to the $\sharp/\flat$-Beilinson-Flach element main conjecture. Since one inclusion of the Greenberg-type main conjecture is known, this then yields half the $\sharp/\flat$-$\sharp/\flat$ main conjecture, from which we can deduce the main theorem.

We first recall in subsection \ref{subsec:sharpflat} the statement of the main conjecture from \cite{shuron}. In subsection \ref{subsec:twovariable}, we recall the analytic theory of Lei and Loeffler, and the algebraic theory due to Kim in the case $a_p=0$, before developing it further to include the case $a_p\neq0$. Subsection \ref{subsec:greenberg} is a short exposition on the Greenberg-type main conjecture, and one inclusion of it, Wan's theorem. In subsection \ref{subsec:beilinsonflach}, we construct $\sharp/\flat$-Beilinson-Flach elements. These are modified versions of the Beilinson-Flach Euler system due to Kings, Loeffler, and Zerbes \cite{klz}. While the appropriate modification in the case $a_p=0$ is due to Wan, a new idea is needed for the case $a_p\neq0$. Once they are defined, we are in good shape to formulate the $\sharp/\flat$ Beilinson-Flach element main conjecture.
\subsection{Statement of the Cyclotomic $\sharp/\flat$ Main Conjecture}\label{subsec:sharpflat} 
We now recall the main conjecture of \cite{shuron} (cf. \cite{kobayashi} for the case $a_p=0$).

\textbf{The analytic side.}
Let $\chi$ be a character of $(\Z/p^{n+1}\Z)^\times$ into $\mu_{p^\infty}$ sending $\gamma$ to $\zeta_{p^n}$, and let $\tau(\chi)$ be the Gau\ss{ }sum $\sum_{a\in(\Z/p^{n+1}\Z)^\times}\chi(a)\zeta_{p^{n+1}}^a$.

\begin{theorem}(Amice and V\'{e}lu, Vi\v{s}ik)
There are $p$-adic power series $L_\alpha(X)$ and $L_\beta(X)$ converging on the open unit disk so that for $\xi\in\{\alpha,\beta\}$,
$$L_\xi(\zeta_{p^n}-1)={p^{n+1} \over \xi^{n+1}}\tau\left(\chi^{-1}\right){L\left(E,{\chi^{-1}},1\right)\over \Omega_E} \text {  if $0 \neq \zeta_{p^n}-1$, and } L_\xi(0)=\left(1-\frac{1}{\xi}\right)^2\frac{L(E,1)}{\Omega_E}.$$
\end{theorem}
These power series are determined uniquely under the condition that they are $O(\log_p(1+X)^\frac{1}{2})$, cf.  \cite[p. 528]{pollack}. (For the definition of growth, see Definition \ref{growth}.)
 We let
$$C_n(X):=\smat{a_p & 1 \\ -\Phi_{p^n}(1+X) & 0},$$
and put $$\Log(X):=\lim_{n\to\infty} C_1(X)C_2(X)\cdots C_n(X)\smat{a_p & 1 \\ -p & 0}^{-(n+2)}\smat{-1 & -1 \\ \beta & \alpha}.$$

\begin{theorem}(\cite{pollack} when $a_p=0$, \cite{surprisingsha} for general $p\mid a_p$)
There is a vector of $p$-adic $L$-functions $(L^\sharp(X),L^\flat(X))\in\Z_p[[X]]^{\directsum 2}$ so that $(L_\alpha(X),L_\beta(X)):=(L^\sharp(X),L^\flat(X))\Log(X).$ Further, we have $(L^\sharp(X),L^\flat(X))\neq(0,0).$
\end{theorem}

Alternatively, the pair of Iwasawa functions is the image of Kato's zeta element under a pair of Coleman maps, i.e. $(L^\sharp(X),L^\flat(X))=(\col_p^\sharp(\bz),\col_p^\flat(\bz))$, where $\bz=(z_n^+)_n\in\bH^1_\Iw(T):=\varprojlim H^1(\Q_{p,n},T)$ is Kato's zeta element \cite[Theorem 12.5]{kato}. The Coleman maps $\col_p^\sharp$ and $\col_p^\flat$ are maps from $\bH^1_\Iw(T)$ to $\Z_p[[X]]$ constructed in \cite[Section 5]{shuron} and will be defined more generally in subsection \ref{subsec:twovariable}, with respect to any if the primes $\gp$ or $\gq$ above $p$.

\textbf{The algebraic side.}
The algebraic object is a modified Selmer group. Let $\star\in\{\sharp,\flat\}$. We put 
$$\Sel^\star(E/\Q_\infty):=\ker\left(\Sel(E/\Q_\infty)\to \frac{E(\Q_{p,\infty})\tensor \Q_p/\Z_p}{E^\star}\right),$$
where $E^\star$ is the exact annihilator of $\ker\col_p^\star$ under the local Tate pairing
$$ \varprojlim_n H^1(\Q_{p,n},T)\times \varinjlim_n H^1(\Q_{p,n},V/T)\to \Q_p/\Z_p.$$

We let $\X^\star(E/\Q_\infty):=\Hom(\Sel^\star(E/\Q_\infty),\Q_p/\Z_p)$, i.e. the Pontryagin dual of $\Sel^\star(E/\Q_\infty)$.

\textbf{The main conjecture.} Putting these together, the Iwasawa main conjecture then states:  
\begin{conjecture}(\cite{kobayashi}, \cite[Main Conjecture 7.21]{shuron})
Choose $\star\in\{\sharp,\flat\}$ so that $L^\star(X)\neq0$. Then 
$$\Char(\X^\star(E/\Q_\infty))=(L^\star(X)).$$
\end{conjecture}
This is the conjecture we want to prove. Half of this conjecture follows from work of Kato:
\begin{theorem}(\cite[Theorem 7.16]{shuron}) Choose $\star\in\{\sharp,\flat\}$ so that $L^\star(X)\neq0$. Then for some integer $n$, we have
$$\Char(\X^\star(E/\Q_\infty))\supseteq(p^nL^\star(X)).$$ If the $p$-adic representation $\GalQ\to\GL_{\Z_p}T$ on the automorphism group of the Tate-module is surjective, we may take $n=0$.

\end{theorem}

\subsection{The $\sharp/\flat$-$\sharp/\flat$ main conjectures for imaginary quadratic fields}
\label{subsec:twovariable}
In this subsection, we recall the analytic theory of Antonio Lei (which generalizes that of Haran/Loeffler (\cite[Theorem 2]{haran} and \cite[Corollary 2]{loeffler}) for the case $a_p=0$), and then generalize the algebraic theory of B.D. Kim (who also worked with the assumption $a_p=0$). We then put the two sides together via a main conjecture.

\textbf{The analytic side.}
Let $\xi,\eta\in\{\alpha,\beta\}$, and denote by $L_{\xi,\eta}(X,Y)$ the $p$-adic $L$-functions of Haran/Loeffler interpolating the following special values at a character $\chi$ of $\Gal(K_\infty/K)$, where the conductor $\gf_\chi$ is of the form $\gp^n\gq^{n'}$ for $n,n'\geq1$:
$$\left(\frac{1}{\xi}\right)^{\ord_\gp\gf_\chi} \left(\frac{1}{\eta}\right)^{\ord_\gq \gf_\chi} \frac{L(E,\chi,1)}{\tau(\chi)|\gf_\chi|\Omega_E^+\Omega_E^-}$$

\begin{theorem}[Lei, \cite{lei}]\label{lei} There exist $L^{\sharp\sharp},L^{\sharp\flat},L^{\flat\sharp},L^{\flat\flat}\in\Lambda_K\tensor\Q$ so that 
$$\links L_{\alpha,\alpha} & L_{\beta,\alpha} \\ L_{\alpha,\beta} & L_{\beta,\beta} \rechts= \Log(Y)^T\links L^{\sharp\sharp} & L^{\flat\sharp} \\ L^{\sharp\flat} & L^{\flat\flat} \rechts\Log(X).$$
\end{theorem}

\begin{proposition} The matrix $\links L^{\sharp\sharp} & L^{\flat\sharp} \\ L^{\sharp\flat} & L^{\flat\flat} \rechts$ is not zero.
\end{proposition}

\begin{proof}Rohrlich \cite[page 1]{rohrlich}, combined with the above interpolation.
\end{proof}

\begin{lemma}[(multivariable $p$-adic Weierstra\ss \vspace{0mm} preparation theorem)]\label{cutelemma}
Let $\O$ be the ring of integers in a $p$-adic field, and let $f\in\O[[X_1,\cdots,X_n]]$. Then for integers $\mu\geq0,\lambda_1\geq0,\cdots \lambda_n\geq0$ we have
$$ f= p^\mu\prod_{i\geq1}^{n-1}(X_i^{\lambda_i}+a_{\lambda_i-1,i}X_i^{\lambda_i-1}+a_{\lambda_i-2}X_i^{\lambda_i-2}+\cdots+a_{1,i}X_i+a_{0,i})U$$
 for a unit $ U\in\O[[X_1,\cdots,X_n]]^\times $ and $ a_{l,i}\in(p,X_1,\cdots,X_{i-1}).$
\end{lemma}

\begin{proof}We need to prove that such $f$ can be written as $f=\prod_{i\geq0}^{n-1}f_i$, with $f_0\in\gm=$(maximal ideal in $\O$), and $f_i\in\gm_i\setminus\gm_{i-1}$ for $i\geq1$, where $\gm_i=$(maximal ideal in $\O[[X_1,\cdots,X_i]]$. But this follows from induction on the one-variable $p$-adic Weierstra\ss \vspace{0mm} preparation theorem \cite[Theorem 7.3]{lang}, or \cite[Corollary 3.2]{venjakob}.
\end{proof}

\begin{proposition}\label{integrality} The functions $L^{\sharp\sharp},L^{\sharp\flat},L^{\flat\sharp},L^{\flat\flat}$ are in $\Lambda_K$.
\end{proposition}
\begin{proof}This follows from interpolation at $X=\zeta_{p^n}-1$ and $Y=\zeta_{p^m}-1$ for large $n,m$, in which the $\mu$-invariant and the $\lambda_1$- and $\lambda_2$-invariant dominate. The positivity of the $\mu$-invariant then follows from \cite[page 375]{manin}, which guarantees integrality of the interpolating values. \end{proof}

\begin{choice}\label{assumption}From \cite[Proposition 6.14]{shuron}, we see that for a choice $\bullet\in \{\sharp,\flat\}$, the one-variable $p$-adic $L$-function $L^\bullet$ is non-zero. Then we pick $K$ so that the $\bullet$ $p$-adic $L$-function $L^{\bullet_K}$ of the elliptic curve $E^{(K)}$ (i.e. the quadratic twist by the character corresponding to $K$) is also not zero. 
\end{choice}
We point out that this is possible in view of a positive proportion of elliptic curves having rank $0$ \cite{manjarul}, combined with a positive proportion satisfying the Birch and Swinnerton-Dyer conjecture \cite{bsz}, and \cite[Table before Proposition 6.14]{shuron}.

\textbf{The algebraic side.}

\begin{lemma}\label{tracecompatibility} We can choose a system of points $c_{(n,m)}\in\Ehat(\gm_{n,m})$ satisfying
$$\tr_{k_{n,m+1}/k_{n,m}}c_{(n,m+1)}=c_{(n,m)}, \text{ and }$$
$$\tr_{k_{n,m}/k_{n-1,m}}c_{(n,m)}=a_pc_{(n-1,m)}-c_{(n-2,m)}.$$
\end{lemma}

This lemma generalizes \cite[Lemma 2.2]{wan} and the constructions in \cite[Proposition 3.12]{kimparity}, both of which treat the $a_p=0$ case. Its proof needs the following definitions.

Given $u\in\O_k^\times$, put $f_u(X):=(u+X)^p-u^p$. Also, we let
$$(x_k,x_{k-1}):=(1,0)A^k\times\frac{1}{p^k} \text{for $k\geq0$},$$
where $A=\smat{a_p & p \\ -1 & 0}$ as in \cite[Definition 2.1]{shuron}. We let $\phi$ be the Frobenius on $k$ and let
$$ f_u^{(k)}:=f_u^{(k)}(X):=f_u^{\phi^{k-1}}\circ f_u^{\phi^{k-2}}\circ\cdots\circ f_u.$$
The logarithm giving rise to our formal group via Honda theory is the power series
$$\log_{f_u}(X):=\sum_{k\geq0}^\infty x_kf_u^{(k)}(X).$$
Now define a sequence $b_i$ via
$$b_1=1,b_2=a_p,\text{ and } b_{i+2}:=a_pb_{i+1}-b_i.$$
This allows us to define the following.
\begin{definition} We put 
$\lambda_{n,u}:=\sum_{i\geq1}^\infty b_i u^{\phi^{-(n+i+1)}}p^{\lceil \frac{i}{2} \rceil}.$
\end{definition}

As in \cite[page 54]{kimparity}, these elements give rise to points \footnote{For BD Kim's notation, the analogue of $\lambda_{n,u}$ is denoted $\lambda_n$ and that of $c_{n,u}$ is denoted $b_n$.} $c_{n,u}\in\Ehat(\gm_{k(\zeta_{p^n}})$ so that
$$\begin{array}{ll}\log_\Ehat (c_{n,u})& = \lambda_{n,u}+\log_{f_u^{\phi^{-n}}}(u^{\phi^{-n}}\cdot (\zeta_{p^n}-1))\\
&=\lambda_{n,u}+\sum_{k\geq0}^\infty x_k\pi_{n-k,u}\end{array}$$
for trace-compatible uniformizers $\pi_{n-k,u}=\pi_{n-k}u$ in $\gm_{k(\zeta_{p^n})},$ where $\pi_{n-k}=\left(\zeta_{p^{n-k}}-1\right).$

\begin{proof}[Proof of Lemma \ref{tracecompatibility}]
To make the first trace relation work, choose trace-compatible $d:=\{d_m\}_m\in\varprojlim_m \O_{k^m}$ as in \cite[Lemma 2.2]{wan}. The calculations in the proof of \cite[Lemma 2.2]{wan} with the $d_m$ expressed as sums of roots of unity then work, once the coefficients of $\log_\Ehat$ are appropriately modified. More precisely, we put $$\log_{\hat{E}}(c_{n,{d_m}})=\sum_i b_i d_m^{\phi^{-(n+i+1)}}+\sum_{k<n}x_k(\zeta_{p^{n-k}}-1)d_m^{\phi^{k-n}}.$$
For the second trace relation, we refer to the calculations in \cite[page 54]{kimparity}. To make them work for the case $a_p\neq0$, note that
$$
\begin{array}{rl}
  \tr_{k(\zeta_{p^n})/k(\zeta_{p^{n-1}})}\log_\Ehat(c_{n,u})&=  \tr_{k(\zeta_{p^n})/k(\zeta_{p^{n-1}})}(\lambda_{n,u}+\pi_{n,u}+\sum_{k\geq1}^\infty x_k \pi_{n-k,u})   \\
  &   =   p\lambda_{n,u}-u^{\phi^{-n}}p+\sum_{k\geq1}^\infty x_k\pi_{n-k,u}\\
  &   =   a_p(\lambda_{n-1,u}+\sum_{k\geq1}^\infty x_{k-1}\pi_{n-k,u})-\lambda_{n-2,u}-\sum_{k\geq2}^\infty x_{k-2}\pi_{n-k,u}\\
  & =a_p\log_\Ehat(c_{n-1,u})-\log_\Ehat(c_{n-2,u}).
\end{array}
$$
\end{proof}

\begin{definition}We define a pairing $P_{(n,m),x}:H^1(k_{n,m},T)\to\Lambda_{n,m}$ by
$$ z \mapsto \sum_{\sigma\in\Gal(k_{n,m}/\Q_p)}(x^\sigma,z)_{n,m}\sigma \text{ for } x \in \F(\gm_{n,m}), \text{ where }$$
$(\quad,\quad)_{n,m}:\F(\gm_{n,m})\times H^1(k_n,T)\to H^2(k_{n,m},\Z_p(1))\cong\Z_p$ is the pairing coing from the cup product.
\end{definition}

\begin{definition}\label{hn}Put $\H_n(X):=-C_1(X)C_2(X)\cdots C_{n-1}(X)\smat{1 & 0 \\ 0 & \Phi_{p^n}(1+X)}$ and define the endomorphism $h_{n,m}$ 
\[
\begin{array}{cccc}
\Lambda_{n,m}\directsum\Lambda_{n,m}  &  \xrightarrow{h_{n,m}}  &   \Lambda_{n,m}\directsum\Lambda_{n,m}\H_n& \subset \Lambda_{n,m}\directsum\Lambda_{n,m} \text{      by}\\
 (a,b) & \mapsto  & (a,b)\H_n  \\
\end{array}
\]

\end{definition}

\begin{remark}The minus sign ensures that our conventions agree with the original ones of Kobayashi, see e.g. \cite[Definition 3.8]{shuron}.
\end{remark}

\begin{proposition}\label{Colemanmapexistence}There exists a unique map $\col_{n,m}$ so that the following commutes:

\[\xymatrixcolsep{11pc}\xymatrix {H^1(k_{n,m},T)\ar@/^13mm/@{-->}[0,2]^{\hspace{2.5mm}\exists !\col_{n,m}}\ar[r]^{(P_{(n,m),c_{n,m}}, P_{(n,m),c_{n-1,m}})}& {\Lambda_{n,m}\oplus \Lambda_{n,m}} &\ar@{_(->}[l]_{h_{n,m}}\dfrac{\Lambda_{n,m} \oplus \Lambda_{n,m}}{\ker h_{n,m}}
}
\]

\end{proposition}
\begin{proof}\cite[Proposition 5.3]{shuron} proves this in the case in which $k^m=\Q_p$, relying on the arguments from \cite[Section 3]{shuron}. These arguments work for our purposes since the pairing $P_{(n,m)},x$ is available in the $k^m$ case as well.
\end{proof}

\begin{lemma} The Coleman maps are compatible in the unramified direction, i.e. the following diagram commutes:
\[\xymatrixcolsep{5pc}\xymatrix {H^1(k_{n,m+1},T)\ar[r]^{\col_{n,m+1}}\ar[d]^{\cor}\ar@{}[dr]|\circlearrowleft& \dfrac{\Lambda_{n,m+1} \oplus \Lambda_{n,m+1}}{\ker h_{n,m+1}}\ar[d]^{\proj}\\
H^1(k_{n,m},T)\ar[r]^{\col_{n,m}}& \frac{\Lambda_{n,m} \oplus \Lambda_{n,m}}{\ker h_{n,m}}}
\]
\end{lemma}
\begin{proof}This follows from the first trace-compatibility stated in lemma \ref{tracecompatibility}.
\end{proof}

\begin{proposition} The Coleman maps are compatible in the totally ramified direction, i.e. the following diagram also commutes:
\[\xymatrixcolsep{5pc}\xymatrix {H^1(k_{n+1,m},T)\ar[r]^{\col_{n+1,m}}\ar[d]^{\cor}\ar@{}[dr]|\circlearrowleft& \dfrac{\Lambda_{n+1,m} \oplus \Lambda_{n+1,m}}{\ker h_{n+1,m}}\ar[d]^{\proj}\\
H^1(k_{n,m},T)\ar[r]^{\col_{n,m}}& \dfrac{\Lambda_{n,m} \oplus \Lambda_{n,m}}{\ker h_{n,m}}}
\]
\end{proposition}
\begin{proof}This is \cite[Corollary 5.6]{shuron} in the case $k^m=\Q_p$. We note that the arguments for this corollary apply in the $k^m$ case as well.
\end{proof}

\begin{lemma}We have $\varprojlim_m\varprojlim_n\frac{\Lambda_{n,m} \oplus \Lambda_{n,m}}{\ker h_{n,m}}\cong \Lambda \directsum \Lambda$.
\end{lemma}
\begin{proof}The arguments in \cite[Proposition 5.7]{shuron} show that $\varprojlim_n\frac{\Lambda_{n,m} \oplus \Lambda_{n,m}}{\ker h_{n,m}}\cong \Lambda \directsum \Lambda$. This is enough since $\Gal(k_{n,m}/k^m)\cong\Gal(\Q_{p,n}/\Q_p)$.
\end{proof}

\begin{definition} We define the pair of Coleman maps as
$$(\col^\sharp,\col^\flat):=\varprojlim_n \varprojlim_m \col_{n,m}: \varprojlim_n \varprojlim_m H^1(k_{n,m},T)\to\Lambda\oplus\Lambda.$$
\end{definition}

\begin{definition}Let $\star\in\{\sharp,\flat\}$. We denote by $E^\star$ the exact annihilator of $\ker \col_\star$ under the Tate pairing
$$\varprojlim_n \varprojlim_m H^1(k_{n,m},T)\times\varinjlim_n \varinjlim_m H^1(k_{n,m},V/T)\to\Q_p/\Z_p.$$
\end{definition}

We denote by $k_{n,m}(\gp)$ (resp. $k_{n,m}(\gq)$) the local field isomorphic to $k_{n,m}$ with initial layer $k_{0,0}$ obtained by completing $K$ at $\gp$ (resp. $\gq$) and climbing up to the appropriate layer of the cyclotomic extension and the unramified tower. 
\begin{definition}We now define the four Selmer groups $\Sel^{\sharp\sharp}(E/K_\infty)$, $\Sel^{\sharp\flat}(E/K_\infty)$, $\Sel^{\flat\sharp}(E/K_\infty)$, and $\Sel^{\flat\flat}(E/K_\infty)$. For $\star,\circ\in\{\sharp,\flat\}$, put
$$\Sel^{\star\circ}(E/K_\infty):=\ker\left(\Sel(E/K_\infty)\to\frac{\varinjlim_m\varinjlim_n H^1(k_{n,m}(\gp),V/T)}{E^\star}\directsum\frac{\varinjlim_m \varinjlim_n H^1(k_{n,m}(\gq),V/T)}{E^\circ}\right).$$ 
\end{definition}

\begin{definition}Given $\star,\circ\in\{\sharp,\flat\}$, put $\X^{\star\circ}:=\Hom(\Sel^{\star\circ}(E/K_\infty),\Q_p/\Z_p)$.
\end{definition}

\begin{proposition}Choose $\star,\circ\in\{\sharp,\flat\}$ so that $L^{\star\circ}(X,Y)\neq0$. Then the Selmer group $\X^{\star\circ}$ is torsion as a $\Lambda$-module.
\end{proposition}
\begin{proof}The arguments in \cite[Theorem 7.14]{shuron} with the appropriate adjustments work.
\end{proof}

Denote by $t$ the anticyclotomic variable. If we pick $K$ and $\bullet \in \{\sharp,\flat\}$ as in our Choice \ref{assumption}, denote by $\X_K^\bullet$ the $\bullet$-Selmer group dual of the elliptic curve $E^{(K)}$.
\begin{proposition}\label{control}We have $\X^{\bullet\bullet}\tensor \Lambda_K/(t)\cong \X^\bullet\oplus\X^\bullet_K$
\end{proposition}
\begin{proof}Tensoring with $\Lambda_K/(t)$ in each step of the construction of $\X^{\bullet\bullet}$ gives an isomorphism between $\X^{\bullet\bullet}\tensor \Lambda_K/(t)$ and the $\bullet$-Selmer group dual in the cyclotomic variable with base field $K$. This $\bullet$-Selmer group dual is isomorphic to $\X^\bullet\oplus\X^\bullet_K$, since $E/K\cong E/\Q\oplus E^{(K)}_\Q$.
\end{proof}

\textbf{The main conjecture.}

\begin{conjecture}\label{sharpflatandsharpflat}
For $\star,\circ\in\{\sharp,\flat\}$ so that $L^{\star\circ}(X,Y)\neq0$, 
$$\Char(\X^{\star\circ}(E/K_\infty)=(L^{\star\circ}(X,Y)).$$
\end{conjecture}

\subsection{A Greenberg-type main conjecture}
\label{subsec:greenberg}

\textbf{The analytic side.}
We let $L_p^{\forall0}\in\Frac\Lambda_K$ be the $p$-adic $L$-function denoted $L_{f,K}'$ in \cite[discussion before Theorem 2.13]{wan}, where $f$ is the modular for associated to $E$. More precisely, define $F_{d,\gp}\in\Zhat_p^{ur}[[U]]$ as $$F_{d,\gp}=\varprojlim_m\sum_{\sigma\in U_\gp/p^mU_\gp} d_m^\sigma\cdot \sigma^2,$$ where the $d_m$ are as in the proof of \cite[Lemma 2.2]{wan}. Let $\phi\in S$, where $S$ is a Zariski dense set of arithmetic points inside $$\{\phi\in \Spec\Lambda:\xi_\phi\text{ is the avatar of a Hecke character of infinity type }\left(\frac{\kappa}{2},\frac{-\kappa}{2}\right) \text{ with } \kappa\geq6\},$$ and $\xi_\phi:=\phi\circ\Psi$. Then $\phi(L_p^{\forall0}\cdot F_{d,\gp})$ interpolates the algebraic parts of (certain) values $L(K,\pi_E,\bar{\xi}_\phi^c,\frac{\kappa}{2}-\frac{1}{2})$ coming from twists of the automorphic representation $\pi_E$ associated to $E$. We refer to \cite[Section 2.3]{wan} for the notation and details since they won't be needed in our paper.


\textbf{The algebraic side.}
The algebraic object is the following ``fine at $\gp$ but empty at $\gq$'' Selmer group.
\begin{definition}
$$\Sel_p^{\forall0}:=\ker\left(H^1(K,T\tensor\Lambda^*(\Psi))\to\prod_{v\nmid p}H^1(K_v,T\tensor\Lambda^*(\Psi))\times H^1(K_\gp,T\tensor\Lambda^*(\Psi))\right).$$
We then put $\X^{\forall0}:=\Hom(\Sel_p^{\forall0},\Q_p/\Z_p)$.
\end{definition}

\textbf{The main conjecture.}
The main conjecture is full equality in the following theorem of Wan \cite[Theorem 1.1]{wanrankin} (see also \cite[Theorem 2.13]{wan}):
\begin{theorem}\label{wanstheorem}Let $E$ have square-free conductor $N$ that has at least one prime divisor $l|N$ not split in $K$, and suppose that $E[p]|_{G_K}$ is absolutely irreducible. Then as (fractional) ideals of $\Lambda_K \tensor \Q_p$, we have $$\Char(\X^{\forall0})\subseteq(L_p^{\forall0}).$$
\end{theorem}

\subsection{The $\sharp/\flat$-Beilinson-Flach main conjectures}
\label{subsec:beilinsonflach}

\textbf{The analytic side.}

\begin{definition}
Fix $0<r<1$. For $f(X)\in \C_p[[X]]$ convergent on the open unit disc of $\C_p$, let  
\[|f(X)|_r:=\sup_{|z|_p<r}|f(z)|_p \]
with normalization $|p|_p=\frac{1}{p}$.
\end{definition}

\begin{definition}\label{growth}Let $f(X), g(X)\in \C_p[[X]]$ converge on the open unit disc of $\C_p$. Then we say that $f(X)$ is $O(g(X))$ if \[\left|f(X)\right|_r\text{  is }O(\left|g(X)\right|_r)\text{  as }r\rightarrow 1^-.\]
If in addition, $g(X)$ is $O(f(X))$, then we say that $f(X)\sim g(X)$.
\end{definition}

\begin{definition} We denote by $D_K^{x,y}$ the distributions of order $x,y$ for non-negative real numbers $x,y$ in the variables $X$ and $Y$, i.e. power series convergent on the open unit disk which are $\O(\log^x)$ in the variable $X$ and $\O(\log^y)$ in the variable $Y$. 
\end{definition}

\begin{convention}Given a column vector of functions $\links f_1 \\ f_2 \rechts$ an a row vector of elements $\left(x_1,x_2\right)$, we denote by $\links f_1 \\ f_2 \rechts \circ \left(x_1,x_2\right)$ the matrix  $\smat{f_1(x_1) & f_1(x_2) \\ f_2(x_1) & f_2(x_2)}$.
\end{convention}




We now construct integral Euler systems $\Delta_\sharp,\Delta_\flat$ out of the Beilinson--Flach Euler systems $\Delta_\alpha, \Delta_\beta$ due to Kings, Loeffler, and Zerbes, constructed from the $p$-stabilizations $f_\alpha$ and $f_\beta$ of $f$ (see \cite[6.4.4]{llzannals} and \cite[Thm A]{lzcoleman}).

\begin{theorem}\cite[Loeffler, Zerbes, Thm A]{lzcoleman}

There exist Euler systems {$\Delta_\alpha, \Delta_\beta \in {H^1\left(K, T\tensor \Lambda_K(\Psi)\right)^2\tensor D_K^{\frac{1}{2},0}}$} and maps $$\calL_\alpha, \calL_\beta: H^1(K_\gq, T\tensor \Lambda_K(\Psi))\tensor D_K^{\frac{1}{2},0}\rightarrow D_K^{\frac{1}{2},0}$$ so that 

$$
\links \calL_\alpha \\ \calL_\beta \rechts \circ \left(\Delta_\alpha, \Delta_\beta\right)=\smat{L_{\alpha\alpha} & L_{\beta\alpha} \\ L_{\alpha\beta} &L_{\beta\beta}}.
$$
\end{theorem}

\begin{proposition}\label{centralprop}There exist integral Euler systems $\Delta_\sharp, \Delta_\flat\in H^1(K, T\tensor \Lambda_K(\Psi))^2$ so that

$$
h\cdot (\Delta_\alpha,\Delta_\beta)=(\Delta_\sharp,\Delta_\flat)\Log(X)
$$
for some element $h\in\Zp[[X,U]].$

\end{proposition}

\begin{lemma}\label{growthofLog}
The entries of $\Log(X)$ are $O(\log_p(1+X)^{\frac{1}{2}})$.
\end{lemma}
\begin{proof}\cite[Proposition 4.20 (Growth Lemma)]{surprisingsha}.
\end{proof}

\begin{definition}We now choose a $\Lambda_K$-basis $(v_1,v_2)$ of $H^1(K_{\gq},T\tensor\Lambda(\Psi))$.
\end{definition}

\begin{definition}Choose $\xi\in\{\alpha,\beta\}$. We choose integral $(v_1,v_2)-$coordinates of the $\gq$-localized Beilinson--Flach element $(\Delta_\xi)_\gq=\loc_\gq(\Delta_\xi)$ by $\delta_{\xi i}$ with $i\in\{1,2\}$ so that
$$((\Delta_\alpha)_\gq,(\Delta_\beta)_\gq)=(v_1,v_2)\smat{\delta_{\alpha1} & \delta_{\beta1}\\ \delta_{\alpha2} & \delta_{\beta2} }\times \frac{1}{f},$$

where

$$\delta_{\xi i}\in D_K^{\frac{1}{2},0} \text{ and } f\in\Z_p[[T]] \text{ is nonzero}.$$
\end{definition}



\begin{proof}[Proof of Proposition \ref{centralprop}]
From Theorem \ref{lei} and Proposition \ref{integrality}, we know that 

\begin{flalign} \Log(Y)^T\goodLs \Log(X)&= \links L_{\alpha\alpha} & L_{\beta\alpha} \\ L_{\alpha\beta} & L_{\beta\beta} \rechts = \hidari \calL_\alpha\\ \calL_\beta \migi \circ \left(\Delta_\alpha,\Delta_\beta\right)_\gq\\&=\hidari \calL_\alpha \\ \calL_\beta \migi \circ \left( v_1,v_2\right) \links \delta_{\alpha1} & \delta_{\beta1} \\ \delta_{\alpha2} & \delta_{\beta2} \rechts \frac{1}{f} \\
&=\links \calL_\alpha(v_1) & \calL_\alpha(v_2) \\ \calL_\beta(v_1) & \calL_\beta(v_2) \rechts \links \delta_{\alpha1} & \delta_{\beta1} \\ \delta_{\alpha2} & \delta_{\beta2} \rechts \frac{1}{f} 
\end{flalign}

The strategy of \cite[Lemma 7.3] {wan} can now be generalized. Thus, there exists an $a\in p\Z_p$ so that

$$\links \delta_{\alpha1} & \delta_{\beta1} \\ \delta_{\alpha2} & \delta_{\beta2} \rechts \left.\Log(X)^{ad}\right|_{Y=a}\in \log_p(X)\times M_2\left(\Frac\Z_p[[X]]\right)$$

by the growth property. Here, $\Log(X)^{ad}$ denotes the adjugate of $\Log(X)$.

We use \cite[Proposition 4.9]{lzzp2} which shows that $$ \links \calL_\alpha(v_1) & \calL_\alpha(v_2) \\ \calL_\beta(v_1) & \calL_\beta(v_2) \rechts \in D_K^{0,\frac{1}{2}},$$ so that the determinant is in $\Q_p$.
\end{proof}

We now choose $h\in\Z_p[[X,U]]$ so that $\links \delta_{\alpha1} & \delta_{\beta1} \\ \delta_{\alpha2} & \delta_{\beta2} \rechts \left.\Log(X)^{ad}\right.\in \log_p(X)\times M_2(\Lambda_K)$.
\begin{corollary}\label{h}The coordinates of the Beilinson-Flach elements are divisible by $\Log(X)$ after correcting by $h$, i.e.
$$h\smat{\delta_{\alpha1} & \delta_{\beta1}\\ \delta_{\alpha2} & \delta_{\beta2} }=\smat{\delta_{\sharp1} & \delta_{\flat1}\\ \delta_{\sharp2} & \delta_{\flat2} }\Log(X)$$
for some $\delta_{*i}\in\Z_p[[X,U,T]]\tensor \Frac \Z_p[[T]]\tensor \Q$ for $\xi\in\{\sharp,\flat\}$ and $i\in \{1,2\}$.
\end{corollary}

\begin{proof} After multiplying by the adjugate of $\Log(X)$, we know that the entries vanish at $p$-power roots of unity. 
\end{proof}

\begin{corollary}The elements $(\Delta_\sharp,\Delta_\flat):=h(\Delta_\alpha,\Delta_\beta)\Log(X)^{-1}$ are well-defined as elements of \\${H^1(\Q_\infty,V_f^*\tensor M(\dg)^*)^{\oplus 2}=H^1_{\Iw}(K_\infty,V_f^*)^{\oplus 2}.}$
\end{corollary}

\begin{proof}This follows from Proposition \ref{integrality}. \end{proof}


\textbf{The algebraic side.}
The objects here are the ``$\sharp/\flat$ at $\gp$ but fine at $\gq$" Selmer groups:
\begin{definition} For $\diamond\in\{\sharp,\flat\}$, put
$$\Sel^{\diamond0}:=\ker\left(H^1(K_\infty,V_f^*\tensor\Lambda_K(\Psi)\tensor(\Lambda_K)^*)\to\prod_{v\nmid p} H^1(I_v,V_f^*\tensor\Lambda_K(\Psi)\tensor(\Lambda_K)^*)\times \frac{H^1_\gp}{E^\circ}\times H^1_\gq\right),$$
where $H^1_\gp:=H^1(K_{\infty,\gp},V_f^*\tensor\Lambda_K(\Psi)\tensor(\Lambda_K)^*)$ and $H^1_\gq$ is defined analogously.
\end{definition}

We then put $$\X^{\diamond0}:=\Hom(\Sel^{\diamond0},\Q_p/\Z_p).$$

\textbf{The main conjecture.}
To set up the main conjecture, we need one more definition on the analytic side (i.e. the side in which the Beilinson-Flach elements live): These are the ``dual $\sharp/\flat$ at $\gp$ but coarse at $\gq$" Selmer groups.
\begin{convention}For $\diamond\in\{\sharp,\flat\}$, we write $\col^{\diamond}_\gp$ for the map $\col^\diamond$ at $\gp$, and $\col^{\diamond}_\gq$ for $\col^\diamond$ at $\gq$. \end{convention}
\begin{definition}We choose $K$ so that we can make a choice $\bullet\in\{\sharp,\flat\}$ as in Choice \ref{assumption} and let
$$\X^{\hat\bullet\forall}=\ker\left(H^1\left(K,T_f\tensor\Lambda_K(-\Psi)\right)\to\prod_{v\nmid p}H^1(I_v,T_f\tensor\Lambda_K(-\Psi))\times\frac{H^1(K_\gp,T_f\tensor\Lambda_K(-\Psi))}{\ker(\col^{\bullet}_\gp)}\right).$$
\end{definition}

\begin{proposition}\label{freeness}Given $\bullet\in\{\sharp,\flat\}$ (and $K$) as in Choice \ref{assumption}, $\X^{\hat\bullet\forall}$ is a free rank one $\Lambda_K$-module.
\end{proposition}

\begin{proof}
Note that we then have $\col_\gp^\bullet$ is non-zero. 
The freeness follows from the irredicbility of the $\mod p$ representation.
First we show that the $\Lambda_K$-rank of $\X^{\hat\bullet\forall}$ is at most one. We know from the construction of Coleman maps that
$$ \rk_{\Lambda_K} \frac{H^1(K_\gq,T\tensor \Lambda_K(-\Psi))}{H^1_\bullet(K_\gq,T\tensor\Lambda_K(-\Psi))} =1. $$ 
But again from the construction of the Coleman maps, we see that the specialization of $$N:= \ker \left( \X^{\hat\bullet\forall}\rightarrow \frac{H^1(K_\gq,T\tensor \Lambda_K(-\Psi))}{H^1_\bullet(K_\gq,T\tensor\Lambda_K(-\Psi))}\right)$$  to the $\Q$-cyclotomic line has rank zero, by equation $(3)$ in \cite[Proof of Theorem 7.14] {shuron}. 
Thus $N$ itself must be of rank zero, whence $\X^{\gp\hat{\bullet}}$ has rank at most one. To see that the rank is at least one, consider the composed map 
$$\X^{\hat\bullet\forall}\longrightarrow\frac{H^1(K_\gq,T\tensor \Lambda_K(-\Psi))}{H^1_\bullet(K_\gq,T\tensor\Lambda_K(-\Psi))}\overset{\col^\bullet}{\longrightarrow}\Lambda_K$$
and note that we have chosen $\bullet$ so that $\col^\bullet(\Delta_\bullet)\neq0$.
\end{proof}

\begin{proposition}Given $K$ and $\bullet\in\{\sharp,\flat\}$ as in Choice \ref{assumption}, $\X^{\bullet0}$ is a finitely generated torsion $\Lambda_K$-module.
\end{proposition}

\begin{proof}
The inclusion $\Sel^{\bullet0}\injects\Sel^{\bullet\bullet}$ gives rise to a surjection
$$\X^{\bullet\bullet}\longrightarrow\X^{\bullet0}\longrightarrow0.$$
To prove that $\X^{\bullet0}$ is finitely generated torsion, it thus suffices to prove that $\X^{\bullet\bullet}$ is.
This follows from Proposition \ref{control}, since any free part of $\X^{\bullet\bullet}$ would cause $\X^{\bullet\bullet}\tensor \Lambda_K/(t)$ to be free, where $t$ is the anticyclotomic variable. Note that since $\col_\gp^\bullet\neq0$, we necessarily have that the corresponding one-variable Coleman map is not zero: $\col^\bullet\neq0$ at $p$.
\end{proof}

\begin{conjecture}{(The $\sharp/\flat$ Beilinson-Flach element  main conjecture)}\label{beilinsonflach}
Given $\circ\in\{\sharp,\flat\}$, we have that $$\Char(\X^{\hat\bullet\forall}/\Delta_\circ\Lambda_K)=\Char(\X^{\bullet0}).$$
\end{conjecture}

\section{Connecting the main conjectures together}
\label{sec:connecting}

\subsection{The $\sharp/\flat$-Beilinson-Flach conjectures are equivalent to the Greenberg-type one}
The aim of this subsection is to prove an equivalence of these main conjectures up to powers of $p$ and certain primes which we term \text{exceptional}. We do this by constructing maps $\Wan^\sharp$ and $\Wan^\flat$ which generalize Wan's map $LOG^+$.
Put $$\H_n(X)=:\smat{\omega_n^\sharp & \Phi_{p^n}\omega_{n-1}^\sharp \\ \omega_n^\flat & \Phi_{p^n}\omega_{n-1}^\flat}.$$
We can describe $\ker(h_{n,m})$, defined in Definition \ref{hn}, explicitly by $\ker(h_{n,m})=(\Lambda_{n,m}\oplus\Lambda_{n,m})H_n^\perp$, where $$\H_n^\perp=X\smat{\omega_n^\flat & -\omega_n^\sharp \\ \Phi_{p^n}\omega_{n-1}^\flat & -\Phi_{p^n}\omega_{n-1}^\sharp}.$$
(See \cite[Lemma 5.8]{shuron} for the $\Lambda_n$-module case. The same arguments apply to the $\Lambda_{n,m}$-module case as well.

\begin{remark} Similarly, the kernel of multiplying by $\H_n^\perp$ is $(\Lambda_{n,m}\directsum\Lambda_{n,m})\H_n$.
\end{remark}

\begin{lemma}\label{annihilation}
We have $(c_{n,m},c_{n-1,m})\H_n^\perp=(0,0)$.
\end{lemma}

\begin{proof}
$\Lambda_{n,m}$-bilinearity of the pair of pairings $(P_{{n,m},c_{n,m}},P_{{n,m},c_{n-1,m}})$.
\end{proof}

\begin{proposition}We can choose $(b_{n,m}^\sharp,b_{n,m}^\flat)\in H^1(k_{n,m},T)^{\directsum 2}$ so that:

$$(i) \quad (b_{n,m}^\sharp,b_{n,m}^\flat)\H_n=(c_{n,m},c_{n-1,m})$$

$$(ii)\quad \tr_{k_{n,m}/k_{n-1,m}}b_{n,m}^{\sharp/\flat}=b_{n-1,m}^{\sharp/\flat}, \text{ and } \tr_{k_{n,m}/k_{n,m-1}} b_{n,m}^{\sharp/\flat}=b_{n,m-1}^{\sharp/\flat}.$$
\end{proposition}

\begin{proof}
The existence of such $(b_{n,m}^\sharp,b_{n,m}^\flat)$ satisfying $(i)$ follows from Lemma \ref{annihilation} and the remark. 

For $(ii)$, the trace compatibility  with respect to $m$ follows from that of $(c_{n,m},c_{n-1,m})$. For that with respect to $n$, note that

$$\tr_{n/n-1}\left((b_{n,m}^\sharp,b_{n,m}^\flat)\H_n\right)=\left(\tr_{n/n-1}b_{n,m}^\sharp,\tr_{n/n-1}b_{n,m}^\flat\right)\H_{n-1}\smat{a_p & p \\ -1 & 0}, $$
$$\text{while } \tr_{n/n-1}\left(c_{n,m},c_{n-1,m}\right)=\left(c_{n-1,m},c_{n-2,m}\right)\smat{a_p & p \\ -1 & 0}.$$
\end{proof}

\begin{definition}
We denote by $H_\sharp(k_{n,m},T)$ the rank one $\Lambda_{n,m}$-submodule of $H^1(k_{n,m},T)$ generated by $b_{n,m}^\sharp$, and define $H^1_\flat(k_{n,m},T)$ analogously.
\end{definition}

\begin{proposition}We have $H_\sharp^1:=\varprojlim_{n,m}H_\sharp^1(k_{n,m},T)=\ker \col^\sharp$ and $H_\flat^1:=\varprojlim_{n,m}H_\flat^1(k_{n,m},T)=\ker \col^\flat$.
\end{proposition}

\begin{proof} For $(b_{n,m}^\sharp,b_{n,m}^\flat)\in H^1(k_{n,m},T)$, we have $\col_{n,m}(b_{n,m}^\sharp,b_{n,m}^\flat)=0$ in $\frac{\Lambda_{n,m}\directsum\Lambda_{n,m}}{\ker \H_n}$ by construction. Thus, we have $H_\sharp^1\subset \ker \col^\sharp$ and $H_\flat^1\subset \ker \col^\flat$. To prove equality, we show that $\frac{\ker \col^\sharp}{H_\sharp^1}=0$. By Nakayama's lemma, it suffices to do this for the Coleman map at level $0$: By \cite[discussion before Definition 7.2]{shuron}, we have $$\col_{0,m}=(\col_{0,m}^\sharp,\col_{0,m}^\flat)=(-a_pP_{0,m,c_{0,m}}+P_{0,m,c_{0,m}},-P_{c_{0,m}})=\left(P_{0,m,b_0^\sharp},P_{0,m,b_0^\flat}\right).$$
But $b_0^{\sharp/\flat}$ annihilate themselves under the cup product, so $\ker\col_{0,m}^\sharp$ is generated by $b_{0,m}^\sharp$. The $\flat$ case is proved similarly.
\end{proof}

\begin{definition}\label{exceptional} The maps $\Wan^\sharp$ and $\Wan^\flat$ map elements of $H_{\sharp/\flat}^1$ to their coordinates, i.e. we put:
$$\Wan^\sharp:H_\sharp^1=\varprojlim_{n,m}H_\sharp^1(k_{n,m},T)\to \Lambda, (f^{\sharp/\flat}_{n,m}\cdot b_{n,m}^\sharp)_{n,m} \mapsto (f^{\sharp/\flat}_{n,m})_{n,m}$$
and define $\Wan^\flat$ similarly.
\end{definition}

\begin{proposition}\label{wantol}Recall we have chosen $\bullet\in\{\sharp,\flat\}$ so that $\Delta_\bullet\neq0$. For such a  choice, we have $$\Wan^\bullet(\Delta_\bullet)=L_p^{\forall0}\times h \times H,$$ where $h$ is the constant chosen just before Corollary \ref{h} and $H \in \Frac \Lambda_\dg$. \end{proposition}
\begin{proof} In this proof, we let $\sum$ denote $\sum_{\sigma_\in\Gamma_{n,\gq}\times U_m},$ where $U_m=U/p^mU$. Denote by $x_{n,m}^{\sharp/\flat}$ the image of $\Delta_{\sharp/\flat}$ in $H^1_{\sharp/\flat}(k_{n,m},T)$. Note that we have \footnote{This is a direct calculation on group-like elements. Note that in the corresponding \cite[Proposition 6.11]{wan}, the bar is missing.} $$\sum \log{x_{n,m}^{\sharp/\flat}}^\sigma \phi(\sigma)=\left(\sum \log  {b_{n,m}^{\sharp/\flat}}^\sigma \phi(\sigma)\right)\overline{\phi}(f_{n,m}^{\sharp/\flat}).$$
We also put $$g_n(\phi):= \sum \log_{\hat E}(c_{n,m})^\sigma \phi(\sigma)=\tau(\phi|_{\Gamma_n})\phi(u)\times \sum_{\upsilon \in U_m} \phi(\upsilon) d_m^\upsilon$$
for a generator $\upsilon \in U_m$.
For convenience, we abbreviate $\sum \left(\log {b_{n,m}^{\sharp}}^\sigma \phi(\sigma),\log {b_{n,m}^{\flat}}^\sigma \phi(\sigma)\right)$ by $\phi(\sum \log \overrightarrow{b})$. 
By combining \cite[Proposition 7.1.4 and Proposition 7.1.5]{lzcoleman} (cf. also \cite[Proposition 7.6]{wan}) with the fact that
$$\phi\left(\sum \log \overrightarrow{b}\right) \H_n(\phi(\gamma)) \times \phi(L_p^{\forall0}) = \sum \left(\log_{\hat E} (c_{n,m})^\sigma, \log_{\hat E}(c_{n-1,m}^\sigma)\right)\phi(\sigma)\times \phi(L_p^{\forall0}),$$ we conclude that at any such $\phi$, there is an element $H' \in \Frac \Lambda_\dg$ so that
$$\phi\left(\frac{-1}{H'}\right)\phi\left(\sum \log\overrightarrow{b}\right)\links \overline{\phi}(f^\sharp) g_n - \phi(L_p^{\forall0}) & 0 \\ 0 & \overline{\phi} (f^\flat) g_n - \phi(L_p^{\forall0}) \rechts \H_n(\phi(\gamma))=(0,0),$$
so that either the first (second) entry of $\phi \left(\sum \log (\overrightarrow b)\right)$ is zero, or the difference between the terms in the first (second) diagonal entry $\overline{\phi}(f^{\sharp/\flat})-\phi(L_p^{\forall0})$ is. Thus, we can interpolate at the $i$th component of the collection of vectors $\phi\left(\sum\log \overrightarrow{b}\right)$ as $n$ and $m$ vary, and encounter infinitely many non-zero values. 
The non-triviality of $L_p^{\forall0}$ allows us to conclude that $f^\bullet$ and $L_p^{\forall0}$ agree at infinitely many points up to the elements $h$ and $H'$. Now when giving $M(\dg)*$ a $\Lambda_\dg$-integral structure, we use different normalizations at $\gp$ and $\gq$ (corresponding to the basis elements $v^+$ and $cv^-$ at the top of \cite[page 39]{wan}), so we are off by another factor $H''\in\Frac\Lambda_\dg$. We let $H:=H'H''$.
\end{proof}



\begin{remark}If one of $\Delta_\sharp$ or $\Delta_\flat$ were zero, then this would imply that the $c_{n,m}^{\sharp/\flat}$ and $c_{n-1,m}^{\sharp/\flat}$ would be linearly dependent, much in the spirit of the ordinary case, in which they are norm-compatible. But this would mean that the arithmetic information of $E$ would be encoded in just one Iwasawa function. Thus, the $a_p\neq0$ case has a bit of an ordinary flavor.
\end{remark}

\begin{definition}We call any height one prime ideal that is a pullback to $\Lambda_K$ of a height one prime ideal of $\Z_p[[\Gal(K_\infty/K^{cyc})]]$ a \textit{strong} prime (ideal).
\end{definition}

\begin{theorem}\label{equivalence}The Greenberg-type main conjecture, i.e. full equality in Theorem \ref{wanstheorem}, is equivalent to the $\sharp/\flat$-Beilinson-Flach main conjecture Conjecture \ref{beilinsonflach}, up to powers of $p$ and strong primes. The same can be said about each of the two halves (i.e. inclusion in just one direction between the algebraic and analytic sides) of the main conjectures.
\end{theorem}

\begin{proof}We have that $(\Delta_\sharp,\Delta_\flat)\in H^1_\sharp\ \directsum H^1_\flat= \ker \col_\gp^\sharp \directsum \ker \col_\gp^\flat,$ since the image of $(\Delta_\alpha,\Delta_\beta)=(\Delta_\sharp,\Delta_\flat)\Log$ at any $\phi\in \Gamma_{n,\gp}\times U_m$ is in the kernel of the pairing $\left(P_{{n,m},c_{n,m}},P_{{n,m},c_{n-1,m}}\right)$ (cf. Proposition \ref{Colemanmapexistence}). We use a four-term exact sequence that follows from the Poitou-Tate exact sequence

$$0 \longrightarrow \X^{\hat\bullet\forall} \longrightarrow H^1_\bullet(K_\gp, T\tensor \Lambda_K(-\Psi))\longrightarrow \X^{\forall0}\longrightarrow \X^{\bullet0}\longrightarrow 0,$$

where the injectivity of the map $\X^{\hat\bullet\forall} \longrightarrow H^1_\bullet(K_\gp, T\tensor \Lambda_K(-\Psi))$ follows from the nontrivially of $\Wan^\bullet$: By our choice of $\bullet\in \{\sharp,\flat\}$, we have $\col^\bullet(\Delta_\bullet)\neq0$, so that $H^1_\bullet\neq0$ as well. We then have an exact sequence

$$0 \longrightarrow \X^{\hat\bullet\forall}/\Delta_\bullet \longrightarrow \Lambda_K\tensor \Q_p/\Wan^\bullet(\Delta_\bullet)\longrightarrow \X^{\forall0}\longrightarrow \X^{\bullet0}\longrightarrow 0.$$

Since characteristic ideals are multiplicative in exact sequences, we conclude that the main conjectures are equivalent to each other up to strong primes, the discrepancy between $\Wan^\bullet(\Delta_\bullet)$ and $L_p^{\forall0}$ (i.e. the factor $H$ in Proposition \ref{wantol}).\end{proof}


\subsection{The $\sharp/\flat$-Beilinson-Flach  conjectures are equivalent to the $\sharp/\flat$-$\sharp/\flat$   conjectures}


\begin{proposition}Let $h \in\Q_p\tensor \Z_p[[X,U]]$ be the element from Proposition \ref{centralprop}. Then there exists a $c\in \Q_p^\times$ so that
$$\smat{\col_\gq^\sharp(\Delta_\sharp) & \col_\gq^\sharp(\Delta_\flat)\\ \col_\gq^\flat(\Delta_\sharp) & \col_\gq^\flat(\Delta_\flat)}=h\cdot c\cdot\smat{L_{\sharp\sharp} & L_{\flat\sharp}\\ L_{\sharp\flat}&L_{\flat\flat}}.$$
\end{proposition}

\begin{proof} It suffices to show that
\begin{flalign}\label{colvalues}\Log(Y)^T\smat{\col_\gq^\sharp(\Delta_\sharp) & \col_\gq^\sharp(\Delta_\flat)\\ \col_\gq^\flat(\Delta_\sharp) & \col_\gq^\flat(\Delta_\flat)}\Log(X)&=\Log(Y)^T\times h \times \goodLs \Log(X)\\
&=\badLs\end{flalign}

Since the entries of $\Log()$ are $\O(\log^{1\over 2})$, we need to prove that (\ref{colvalues}) holds at all cyclotomic points in the variables $X$ and $Y$, i.e. at $X=\zeta_{p^n}-1$ and $Y=\zeta_{p^{n'}}-1$. 

For evaulation at $Y=\zeta_{p^{n'}}-1$, we can relate $\Log(\zeta_{p^{n'}}-1)^T\left.\smat{\col_\gq^\sharp(\Delta_\sharp) & \col_\gq^\sharp(\Delta_\flat)\\ \col_\gq^\flat(\Delta_\sharp) & \col_\gq^\flat(\Delta_\flat)}\right|_{Y=\zeta_{p^{n'}}-1}$ to the pairing $P_n^{(1 \text{ or } 0)}(\Delta_{\sharp/\flat})$ by the proof of \cite[Proposition 6.5]{shuron}. Using the relation $$(\Delta_\alpha,\Delta_\beta)=(\Delta_\sharp,\Delta_\flat)\Log(X),$$ the claim then follows from
 Proposition \ref{Colemanmapexistence}, the identity
$$P_{n,m,x}(z)=\left(\sum_{\sigma\in\Gamma_n\times U_m}\log_\gp(x^\sigma)\sigma\right)\left(\sum_{\sigma\in\Gamma_n\times U_m}\exp^*_\gq(z^\sigma)\sigma^{-1}\right) $$
and the explicit reciprocity law of Loeffler and Zerbes (\cite[Theorem 7.1.4 and Theorem 7.1.5]{lzcoleman}), which says that up to a constant in $\overline{\Q}_p^\times$,
$$\pi_{\calF_\dg^+}\hidari\pi_\alpha \\ \pi_\beta\migi\circ\exp^*\circ\phi\circ\left(\Delta_\alpha,\Delta_\beta\right)=\epsilon(\chi_{\phi}^{-1})
\smat{ \alpha^{-n}\eta_{f_\alpha}^*\tensor \phi(\omega_\dg^*)& 0 \\ 0 & \beta^{-n}\eta_{f_\beta}^*\tensor \phi(\omega_\dg^*)}\links\phi(L_{\alpha,\alpha}) & \phi(L_{\beta,\alpha})\\\phi(L_{\alpha,\beta}) & \phi(L_{\beta,\beta})\rechts,$$
where $\phi$ corresponds to a primitive character of $\Gamma_{n,\gq}\times U_m$ with $n>0$. Here, $\eta_{f_\alpha}^*$ and $\eta_{f_\beta}$ are the projections along the $\alpha$- and $\beta$- eigencomponents of $\eta_f^*$, which are the duals of the element used in \cite{klz}. Similarly, we $\omega_\dg^*)$ is the dual of the corresponding element in \cite{klz}, cf. \cite[Discussion in the middle of page 37]{wan} . Finally, we put
$\epsilon(\chi_{\phi}^{-1})=\int_{\Q_p^\times}\chi_{\phi}^{-1}(x^{-1})\lambda(x)dx,$ where $\lambda$ is (any) additive character so that $\ker(\lambda)=\Z_p$ and $\lambda\left(\frac{1}{p^n}\right)=\zeta_{p^n}$, and $\chi_\phi$ denotes composition of $\Psi$ and $\phi$.

The constant $c\in \Q_p^\times$ is chosen as in \cite[Corollary 7.5]{wan}, where it is denoted $H_1$.

\end{proof}

\begin{theorem}\label{sfbf}The $\sharp/\flat$-$\sharp/\flat$ main conjecture (Conjecture \ref{sharpflatandsharpflat}) is equivalent to the $\sharp/\flat$-Beilinson-Flach main conjecture, Conjecture \ref{beilinsonflach} up to strong primes.
\end{theorem}

\begin{proof}From the Poitou--Tate sequence, we have 
$$\X^{\hat\bullet\forall}\longrightarrow \frac{H^1(K_\gq,T\tensor \Lambda_K(-\Psi))}{H^1_\bullet(K_\gq,T\tensor\Lambda_K(-\Psi))}\longrightarrow \X^{\bullet\bullet}\longrightarrow\X^{\bullet0}\longrightarrow 0.$$
Recall we have chosen $K$ so that neither $\col^\bullet$ nor $\Delta_\bullet$ are zero. Since we thus have $\col^\bullet(\Delta_\bullet)\neq0$ and $\X^{\hat\bullet\forall}$ is free of rank $1$ (cf. Proposition \ref{freeness}), the first map in the above exact sequence being injective as a $\Lambda_K$-module is equivalent to its non-triviality.

thus, taking quotients by the signed Beilinson--Flach elements rec. their images, we have the following exact sequence:

$$ 0 \longrightarrow \X^{\hat\bullet\forall}/\Delta_\bullet\longrightarrow \Lambda_K/ \col^\bullet(\Delta_\bullet) \longrightarrow \X^{\bullet\bullet}\longrightarrow \X^{\bullet0}\longrightarrow 0$$
Now by multiplicativity of characteristic ideals, we get the desired equivalence of main conjectures.
\end{proof}

\subsection{Completing the argument}
We want to carry over the inclusion of Theorem \ref{wanstheorem}. First,  we have the following consequence of Theorems \ref{equivalence} and \ref{sfbf}:

\begin{proposition}Under the assumptions of Theorem \ref{wanstheorem}, we have $\text{length}_P\X^{\bullet\bullet}\geq \ord_P L^{\bullet\bullet}$ for any height one prime $P$ that is neither $(p)$ nor the pullback of the augmentation ideal of $\Z_p[[K_\infty/K_{cyc}]]$. 
\end{proposition}
\begin{proof} For the strong primes, note that the discussion of \cite[Lemma 8.4]{wan} applies since it does not assume $a_p=0$. We thus conclude that this direction of the Greenberg-type main conjecture is true at such primes. For the four term exact sequence in the proof of Theorems \ref{equivalence} and \ref{sfbf}, we thus have at any such strong prime $P$ that:

$$\text{length}_P \Char \X^{\forall0}\geq \ord_P L_p^{\forall0} \geq \ord_P\Wan^\bullet(\Delta_\bullet), $$

which by Theorem \ref{equivalence} implies that
$$ \text{length}_P \Char \X_{\bullet0}\geq \text{length} \frac{\X^{\hat\bullet \forall}}{\Delta_\bullet},$$

which in turn implies by Theorem \ref{sfbf} that
$$ \text{length}_P \Char \X^{\bullet\bullet} \geq \ord_P \col_\bullet(\Delta_\bullet)\geq \ord_P L^{\bullet\bullet}.$$
\end{proof}

This is The goal of this subsection is to take care of the prime $p$ and the exceptional prime ideals.

For the prime $p$, we first prove that the $\mu$-invariants of the $p$-adic $L$-functions $L^{*\circ}$ for $\circ,*\in\{\sharp,\flat\}$ are zero.
\begin{proposition}\label{muinvar}Suppose we have chosen the right $K$ and $\bullet\in\{\sharp,\flat\}$ as in Choice \ref{assumption} Suppose $N$ is a product of distinct unramified primes, an odd number of which are inert, so that for any such inert $q|N$, $\overline{\rho}|_{G_q}$ is ramified. Then $\ord_{(p)}L^{*\circ}=0$ for $\circ,*\in\{\sharp,\flat\}$ chosen so that $L^{*\circ}\neq0$.\end{proposition}
\begin{proof}By interpolation, it suffices to prove this for $L^{\bullet\bullet}$. As in \cite[Prop. 4.6]{wan}, this follows from the vanishing of the $\mu$-invariants for the specialization of $L^{\bullet\bullet}$ to the anticyclotomic line. Note that the assumptions of \cite{pollackweston} are not necessary in our case, cf. the discussion in \cite[Proposition 8.6]{wan} following a result of Chan-Ho Kim \cite{chanho}. In the case $a_p=0$, this follows from \cite[Theorem 2.5]{pollackweston}. In the case $a_p=0$, the anticyclotomic $p$-adic $L$-functions $L_{anti}^\pm$ were constructed by combining sequences $\{L_n\}_{n\geq1}$ satisfying certain three-term relations with Kobayashi's construction of the $\pm$-Coleman maps \cite[Section 8]{kobayashi}. The sequences $\{L_n\}$ have no $a_p=0$ assumption, so that we can combine their construction (word for word in the same way) with that of the $\sharp/\flat$-Coleman maps in \cite[Section 5]{shuron} to arrive at anticyclotomic $p$-adic $L$-functions $L_{anti}^{\sharp/\flat}$. The arguments in \cite[Proof of Theorem 2.5]{pollackweston} then show that $\mu(L_n)=0$ for $n \gg 0$. But \cite[discussion at the end of page 165]{bernadette} shows that the $\mu_\sharp/\mu_\flat$-invariants of Perrin-Riou are then $0$. By \cite[Corollary 8.9]{surprisingsha}, we conclude that the $\mu$-invariant of $L_{anti}^\bullet$ is zero.
\end{proof}

Recall that we denoted by $L^{(K)\bullet}$ the $\bullet$ $p$-adic $L$-function for the elliptic curve $E^{(K)}$.
\begin{corollary}\label{gold} Under the assumptions of Proposition \ref{muinvar} and the further assumption that the anticyclotomic root number is $-1$ (so that the restriction of $L^{\bullet\bullet}$ to the anticyclotomic line is nonzero), we have $$\Char(\X^\bullet)\Char(\X^\bullet_K)\subseteq(L^\bullet L^{(K)\bullet})$$.
\end{corollary}
\begin{proof}First, note that $L_{\bullet\bullet}|_{t=0}=L_\bullet L_\bullet^{(K)}$. This can be seen as follows:

From the interpolation property, we have that restriction to the cyclotomic line (setting ${X=Y}$) yields

$$\badLs=\hidari L_\alpha \\ L_\beta \migi \left( L_\alpha^{(K)}, L_\beta^{(K)}\right),$$ where $L_*^{(K)}$ is the $L$-function of $E$ by the quadratic twist by $K$.

But we know that $$\hidari L_\alpha \\ L_\beta \migi = \Log^T\hidari L_\sharp \\ L_\flat \migi$$ and $$ \left( L_\alpha^{(K)}, L_\beta^{(K)}\right)\Log.$$ Restricting $\Log(Y)^T\goodLs\Log(X)$  and factoring out the $\Log$ matrices then yields
$$\left.\goodLs\right|_{t=0}=\hidari L^\sharp \\ L^\flat \migi \left(L_\sharp^{(K)}, L_\flat^{(K)}\right).$$
No exceptional prime divides $L^{\bullet\bullet}$ since the pullback of the augmentation ideal of the anticyclotomic line does not include $L^{\bullet\bullet}$, which follows from the fact that its specialization to the cyclotomic line is not zero by Choice \ref{assumption}. The corollary up to the prime $p$ thus follows from combining Theorem \ref{wanstheorem} and the equivalences of main conjectures Theorem \ref{equivalence}, and for the correct power of $p$ we use Proposition \ref{muinvar}.
\end{proof}

\section{Conclusion}\label{sec:proofs}

We are now ready to prove our main theorem.

\begin{theorem}Let $E/\Q$ be an elliptic curve and $p>2$ a prime of supersingular reduction. Assume that $E$ has square-free conductor. Choose $\bullet\in\{\sharp,\flat\}$ so that $L^\bullet\neq0$. Then $L^\bullet$ is a characteristic power series of the Iwasawa module $\Hom(\Sel^\star(E/\Q_\infty),\Q_p/\Z_p)$.
\end{theorem}

\begin{proof}From \cite[Theorem 7.16]{shuron}, we know that $\Char(\X^\bullet)\supseteq(L^\bullet)$ and $\Char(\X^\bullet_K)\supseteq(L^{(K)\bullet})$ $\Gal(\overline{\Q}/\Q)$ surjects onto $\Aut E[p]$. Since $N$ is square-free, we only need $E[p]|_{\Gal(\overline{\Q}_p/\Q_p)}$ to be absolutely irreducible, as explained in \cite{skinner}[discussion after Theorem 2.5.2]. This irreducibility is part of a theorem of Fontaine \cite{fontaine}. For the entire theorem, see \cite[Theorem 2.6]{edixhoven}. A proof can be found in \cite[Section 6.8]{edixhoven}. Now we can choose our auxiliary $K$ so that it satisfies the assumptions for Proposition \ref{muinvar} \cite[Proof of Theorem 1.3, pg. 19]{wan}, Choice \ref{assumption}, and the assumption on the anticyclotomic root number (cf. Corollary \ref{gold}). By the square-free assumption of $N$ and Ribet's level-lowering theorem \cite[Theorem 8.2]{ribet}, $\left.E[p]\right|_{\Gal(\overline\Q_p/\Q_p)}$ is ramified at some prime $q$. Now choose $K$ so that $p$ and all prime divisors of $N$ different from $q$ split while $q$ is inert. All this only inflicts congruence conditions on roc choice of $K$, so that Corollary \ref{gold} till holds. Thus, we have $\Char(\X^\bullet)\Char(\X^\bullet_K)=(L^\bullet)(L^{(K)\bullet})$, whence $\Char(\X^\bullet)=(L^\bullet)$ for our chosen $K$.
\end{proof}

We now prove the corollaries concerning the leading term formula in the Birch and Swinnerton-Dyer conjecture.

\begin{theorem}(Birch and Swinnerton-Dyer formula at $p$ for rank $1$)

 Let $E$ be an elliptic curve with square-free conductor $N$ with supersingular reduction at $p\neq2$. 
Suppose that ${\ord_{s=1}L(E,s)=1}$. Then 

$$\left|\frac{L'(E,1)}{\Reg(E/\Q)\Omega}\right|_p=\left|\frac{\#\Sha(E/\Q)\prod_l c_l}{\#E(\Q)^2_\tor}\right|_p$$

\end{theorem}

\begin{proof}The Iwasawa main conjecture is the only missing ingredient in the discussion of \cite[Proof of Corollary 1.3]{kobayashigz}.
\end{proof}

\begin{theorem}(Birch and Swinnerton-Dyer formula at $p$ for rank $0$)\label{bestthingwehavedone!}

 Let $E$ be an elliptic curve with square-free conductor $N$ with supersingular reduction at $p\neq2$. Assume that $L(E,1)\neq0$. Then 

$$\left|\frac{L(E,1)}{\Omega}\right|_p=\left|\#\Sha(E/\Q)\prod_l c_l\right|_p,$$
\end{theorem}

Here, $\Sha(E/\Q)$ is the Tate--Shafarevich group of $E/\Q$, $c_l$ are the Tamagawa numbers, and $\Omega$ is the N\'{e}ron period. The proof of this theorem will occupy the final subsection of this paper.
Combined with the corresponding result in the ordinary case \cite[Theorem 3.35]{SU}, this gives the leading term formula up to powers of $2$ and bad primes.

\begin{subsection}{Proof of the Birch and Swinnerton-Dyer formula at $p$ in the rank $0$ case}
We now prove Theorem \ref{bestthingwehavedone!}.
We denote by $E_{\infty,p}^\sharp$ (resp. $E_{\infty,p}$) the exact annihilator of the Coleman maps in the $\Q$-cyclotomic case $\ker \col_p^\sharp$ (resp. $\ker \col_p^\flat$) under the local Tate pairing.
Similarly, we let $E_{0,p}^\sharp$ (resp. $E_{0,p}^\flat$) be the exact annihilator of $\ker \col_{p,0}^\sharp$ resp. $\ker \col_{p,0}^\flat$. See \cite[Definition 7.2]{shuron}\footnote{The maps $\col_{p,0}$ are denoted simplu $\col_0$ in \cite{shuron}.}. Note that $E_{0,p}^\sharp=E_{0,p}^\flat=E(\Q_p)\tensor \Q_p/\Z_p$. We let
$$ \G^\bullet(\Q_n)=\im\left(H^1(\Q_n,E[p^\infty])\longrightarrow\frac{H^1(\Q_{n,p},E[p^\infty])}{E^\bullet{n,p}}\times\prod_{v\neq p} \frac{H^1(\Q_{n,v},E[p^\infty])}{E(\Q_{n,v})\tensor \Q_p/\Z_p}\right)$$ for $n=0$ or $n=\infty$. We denote $\Gal(\Q_\infty/\Q)$ by $\Gamma$.

We then have a commutative diagram with exact sequences with right vertical map $g$:

$\begin{CD}0 @>>> \Sel(E,\Q)_p @>>> H^1(\Q,E[p^\infty])@>>>\calG^\bullet(\Q)@>>>0\\
@. @VVV @VVV @VVgV
\\
0@>>>\Sel^\bullet(E,\Q_\infty)_p^{\Gamma}@>>>H^1(\Q_\infty,E[p^\infty])^{\Gamma}@>>>\calG^\bullet(\Q_\infty)^{\Gamma}
\end{CD}$

For convenience, we let $\calP_E(\Q_n):=\frac{H^1(\Q_{n,p},E[p^\infty])}{E^\bullet{n,p}}\times\prod_{v\neq p} \frac{H^1(\Q_{n,v},E[p^\infty])}{E(\Q_{n,v})\tensor \Q_p/\Z_p}.$

We denote the map $\calP_E(\Q_0)\longrightarrow \calP_E(\Q_\infty)$ by $r$.

\begin{lemma}\label{1}We have:
$$\left|\Sel^\bullet(E,\Q_\infty)_p^{\Gal(\Q_\infty/\Q)}\right|=\frac{\left|\Sel(E,\Q)_p\right|}{\#E(\Q)_p}\times\left|\ker g\right|$$
\end{lemma}
\begin{proof} The proof of \cite[Lemma 4.3]{greenberg} with ``$\Sel$" replaced by ``$\Sel^\bullet$" works, taking into account Lemma \ref{3} below.
\end{proof}

\begin{lemma}\label{2}Denote by $c_l^{(p)}$ the $p$-component of the Tamagawa number $c_l$. We then have: 
$$\frac{1}{\left|\left(\Sel^\bullet(E,\Q_\infty)_p\right)_{\Gal(\Q_\infty/\Q)}\right|}=\frac{\displaystyle{\prod_{l \text{ bad}}c_l^{(p)}}}{\#E(\Q)_p}\times\frac{1}{\left|\ker g\right|}$$
\end{lemma}
\begin{proof}
The lemma with the Tamagawa factors ``$\displaystyle{\prod_{l \text{ bad}} c_l^{(p)}}$" replaced with ``$\left|\ker(r)\right|$" follows from the proof of \cite[Lemma 4.7]{greenberg}, again with ``$\Sel$" replaced with ``$\Sel^\bullet$". Thus, it remains to prove that
$$\left|\ker (r)\right|= \displaystyle{\prod_{l \text{ bad}} c_l^{(p)}},$$ which we do prime by prime. Let $r_v$ be the component of $r$ at $v$. 

If $v \neq p$ is a good reduction prime, we know that $\ker (r_v)$ is trivial by \cite[Proof of Lemma 3.3]{greenberg}.

If $v=l$ is a bad reduction prime, the discussion after \cite[Lemma 3.3]{greenberg} shows that $\left|\ker(r_l)\right|=c_l^{(p)}$.

It thus remains to treat the case $v=p$. (Note that the methods in \cite[Lemma 3.4]{greenberg} do not apply since ordinarity is a key assumption there.) We want to show that the map
$$\frac{H^1(\Q_p,E[p^\infty])}{E(\Q_p)\tensor\Q_p/\Z_p}\overset{r_p}\longrightarrow\frac{H^1(\Q_{p,\infty},E[p^\infty])}{E^\sharp_{\infty,p}}$$ is injective.

As in \cite[Proof of Proposition 9.2]{kobayashi}, we consider the following commutative diagram with exact sequences, were $X$ is the cyclotomic variable:

$\begin{CD}@. \frac{\ker \col^\bullet_p}{X}@>>> \frac{\mathbb{H}^1_{\text{Iw}}}{X}@>>>\frac{\left(\mathbb{H}^1_{\text{Iw}}/\ker\col^\bullet_p\right)}{X}@>>>0\\
@.@VVV@VVV@VVV\\
0@>>>\ker\col^\bullet_{p,0}@>>>H^1(k_0,T)@>>>\frac{H^1(k_0,T)}{\ker \col_{p,0}^\bullet},
\end{CD}$

where $\mathbb{H}^1_{\text{Iw}}=\varprojlim_n H^1(\Q_{p,n},T)$.

By construction, the right vertical map is an isomorphism. The middle vertical map is a surjection by \cite[Lemma 2.3]{shuron}. The snake lemma thus shows that the left vertical map is surjective. Taking Pontryagin duals, we see that $r_p$ is injective, as claimed.
\end{proof}

\begin{lemma}\label{3}
$\Sel^\bullet(E,\Q_\infty)_p^{\Gal(\Q_\infty/\Q)}$ is finite if $L(E,1)\neq0$.
\end{lemma}
\begin{proof}It suffices to prove that $\X^\bullet/X\X^\bullet$ is finite. This follows from the small control theorem below, which implies that $\X^\bullet/X\X^\bullet$ is finite if and only if $\X_0^\bullet/X\X_0^\bullet$ is, where $\X_0^\bullet=\Hom(\Sel^\bullet(E,\Q),\Q_p/\Z_p)$. $\X_0^\bullet/X\X_0^\bullet$ being finite is automatic by our assumption. One can also show that $\X^\bullet/X\X^\bullet$ is finite directly by observing that $L(E,1)\neq0$ implies that $f^\bullet(0)\neq0$ using \cite[table before Proposition 6.14]{shuron}.
\end{proof}

\begin{lemma}(Small Control Theorem) The natural morphism 
$\X^\bullet/X\X^\bullet\longrightarrow\X_0^\bullet/X\X_0^\bullet$ is surjective and has finite kernel.
\end{lemma}

\begin{proof}The proof of \cite[Theorem 9.3]{kobayashi} with $n=0$ with the adjustment that $r_p$ is injective (cf. Lemma \ref{2}) works. (Note that in \cite[diagram on top of page 27]{kobayashi}, the terms on the right should be $\calP_E(\Q)$ and $\calP_E(\Q_\infty)$ rather than $\displaystyle{\prod_v} \mathscr{H}^+(K_{n,v})$.)
\end{proof}

\begin{lemma}\label{4}
Let $f^\bullet$ be a generator of the characteristic ideal of $\X^\bullet$. Then

$$\left|f(0)\right|_p=\frac{\left|\Sel^\bullet(E,\Q_\infty)_p^{\Gal(\Q_\infty/\Q)}\right|}{\left|\left(\Sel^\bullet(E,\Q_\infty)_p\right)_{\Gal(\Q_\infty/\Q)}\right|}.$$
\end{lemma}
\begin{proof}This is a general fact about finitely generated $\Lambda$-modules, using the assumption that $\Sel^\bullet(E,\Q_\infty)_p^{\Gal(\Q_\infty/\Q)}$ is finite, which Lemma \ref{3} guarantees. See \cite[Lemma 4.2]{greenberg}
\end{proof}

\begin{proof}(Proof of Theorem)
The theorem now follows using Lemma \ref{4} and multiplying the terms in Lemma \ref{1} and Lemma \ref{2}, and noting that since $E[p]$ is irreducible as a Galois representation, we necessarily have $\left|E(\Q)_p\right|=1$. The fact that we can replace $\Sel$ by $\Sha$ then follows from \cite[Sledstvie 2]{kolyvagin}.
\end{proof}
\end{subsection}

\begin{description}
\item[Acknowledgments]
\small We are grateful to Xin Wan for being very pessimistic about the feasibility of the $a_p\neq0$ case, and for patiently answering numerous questions. We also thank Christopher Skinner for helpful discussions, and  Richard Taylor who suggested doing a small seminar on this topic. Our thanks also go to Masato Kurihara, Antonio Lei, Yuichi Hirano, Wei Zhang, and Masataka Chida for pointing out inaccuracies in an earlier version of the manuscript. We thank Tadashi Ochiai for improving an earlier version, for a helpful piece of advice about the exposition, and for pointing out some errors in said earlier version of this work.\end{description}

\end{document}